\documentclass[11pt,letterpaper]{amsart}

\usepackage{float}
\usepackage{euscript,verbatim}
\usepackage{graphicx}
\usepackage[usenames]{color}
\usepackage[colorlinks,linkcolor=red,anchorcolor=blue,citecolor=blue]{hyperref}
\usepackage{amsmath}
\usepackage{amsthm}
\usepackage[all]{xy}
\usepackage{amssymb} 
\usepackage{epstopdf}

\newtheorem{thm}{Theorem}[section]
\newtheorem*{question}{Question}

\newtheorem*{thm_a}{Theorem A}
\newtheorem*{thm_b}{Theorem B}
\newtheorem{prop}[thm]{Proposition}
\newtheorem{lem}[thm]{Lemma}
\newtheorem{cor}[thm]{Corollary}

\newtheorem{ex}[thm]{Example}

\theoremstyle{definition}
\newtheorem{df}[thm]{Definition}
\newtheorem{rem}[thm]{Remark}

\def\N{\mathbb{N}}
\def\Z{\mathbb{Z}}
\def\R{\mathbb{R}}

\def\ocap{\text{\rm ocap}}
\DeclareMathOperator{\cocap}{\ocap_\sharp}
\DeclareMathOperator{\htop}{h_{top}}
\DeclareMathOperator{\h}{h}
\DeclareMathOperator{\id}{Id}
\def\AA{\mathcal{A}}
\def\CC{\mathcal{C}}

\def\SS{\mathcal{S}}
\def\GG{\mathcal{G}}

\def\In{\text{\rm Int}^{\Phi}}
\def\Int{\text{\rm Int}}

\def\Pa{\partial^{\Phi}}
\def\q{{\bf q}}
\DeclareMathOperator{\per}{Per}
\DeclareMathOperator{\diam}{diam}
\DeclareMathOperator{\dist}{dist}
\makeatletter 

\title[Symbolic extensions for expansive topological flows]{Strongly isomorphic symbolic extensions for expansive topological flows}
\author{Yonatan Gutman}

\address{Institute of Mathematics, Polish Academy of Sciences, ul. \'Sniadeckich 8, 00-656 Warszawa, Poland}
\email{y.gutman@impan.pl}

\author{Ruxi Shi}
\address{Sorbonne Universit\'e, LPSM, 75005 Paris, France}
\email{ruxi.shi@upmc.fr}

\thanks{Y.G. was partially supported by the National Science Centre (Poland) grant
2020/39/B/ST1/02329. R.S was partially supported by Fondation Sciences Mathématiques de Paris.}

\begin{document}
\keywords{symbolic extension, topological
flow, strongly isomorphic, expansive,  small flow boundary, cross-section.}
\subjclass[2020]{37B10,
37A35.}	
	\maketitle

\begin{abstract}

In this paper, we prove that finite-dimensional topological flows without fixed points and having a countable number of periodic orbits, have the small flow boundary property. This enables us to answer positively  a question of Bowen and Walters from 1972: Any expansive topological flow has a strongly isomorphic symbolic flow extension, i.e. an extension by a suspension flow over a subshift. Previously Burguet had shown this is true if the flow is assumed to be $C^2$-smooth.
\end{abstract}


\tableofcontents

\section{Introduction}

Symbolic coding of dynamical systems $(X,T)$ in the form of measurable factor maps into shift spaces over finite alphabets $(X,T)\rightarrow (Y,\sigma)\subset (\{1,2,\ldots, a\}^\Z,\sigma)$ has played a prominent role in the theory of dynamical systems since its inception (\cite{hadamard1898surfaces,morse1921recurrent,milnor1988iterated}). Requiring the factor maps to be continuous is usually impossible. It is however meaningful to look for a symbolic system $(Y,\sigma)\subset (\{1,2,\ldots, a\}^\Z,\sigma)$ for which the original system occurs as a continuous factor $(Y,\sigma)\rightarrow (X,T)$. In other words this form of ``digitization" corresponds to symbolic extensions. However as $(Y,\sigma)$ is an extension of $(X,T)$ it is necessarily ``more complex" than $(X,T)$. Viewing $(Y,\sigma)$ as a model of $(X,T)$ one strives to minimize this ``complexity gap". This may be formalized mathematically by various conditions such as requiring the extension to be principal or strongly isomorphic\footnote{See Definition \ref{df:principal}}. The associated theory for $\Z$-systems is deep and extensive (\cite{BDz2004, downarowicz2005entropy,D11}). Recently a symbolic extension theory for $\mathbb{R}$-systems, i.e.\ \textit{topological flows} has been put forth (\cite{burguet2019symbolic}). The present paper is a further contribution in this direction, specifically  to the theory of symbolic extensions of expansive topological flows. Our  main tool is the \textit{small flow boundary property}.
 
 As a dynamical analog of Lebesgue covering dimension zero, the {\it small boundary property} for a $\mathbb{Z}$-system $(X,T)$ was introduced by Shub and Weiss in \cite{SW} who investigated the question under which conditions a given $\mathbb{Z}$-system has factors with strictly lower entropy. Later it was realized that the small boundary property has wider applicability. Notably Lindenstrauss and Weiss \cite{LW} showed that a $\mathbb{Z}$-system which has the small boundary property must have mean dimension zero. Moreover a $\mathbb{Z}$-system with the small boundary property has a zero-dimensional strongly isomorphic extension (\cite[p. 4338]{burguet2019symbolic} based on \cite{downarowicz2005entropy}). From \cite{L95} it follows that a finite-dimensional $\mathbb{Z}$-system without periodic points has the small boundary property\footnote{More generally, from \cite[Theorem 3.3]{L95} it follows that an infinite $\mathbb{Z}$-system ($|X|~=~\infty)$ with a finite number of periodic points has the small boundary property.}.

 Burguet \cite{burguet2019symbolic} introduced the {\it small flow boundary property} for topological flows as an analog to the small boundary property. He showed that flows with the small flow boundary property admit strongly isomorphic zero-dimensional extensions
 and gave  necessary and sufficient conditions for the existence of \textit{symbolic} extensions\footnote{A topological flow is said to admit a symbolic extension, respectively a zero-dimensional extension, if it has an extension by a suspension
flow over a subshift, respectively a zero dimensional system, with a positive continuous roof function.} for such  flows in terms of the existence of \textit{superenvelopes} (\cite[Theorem 3.6]{burguet2019symbolic}).
This can be seen as a certain generalization of the Boyle-Downarowicz symbolic extension entropy theorem (\cite{BDz2004}).
 Burguet \cite{burguet2019symbolic} showed that a $C^2$-flow without fixed points\footnote{Flows without fixed points are known as \emph{regular flows} but we will not use this terminology in this paper.} and such that for any $\tau>0$, the number of periodic orbits of period less than $\tau$ is finite has the small flow boundary property. In our main theorem we manage to remove the smoothness  assumption: 
\begin{thm_a}
Let $X$ be a compact finite-dimensional space.
    Let $\Phi$ be a topological flow on $X$ without fixed points, having a countable number of periodic orbits. Then  $(X, \Phi)$ has the small flow boundary  property.
\end{thm_a}

Expansive $\mathbb{Z}$-systems were introduced as early as 1950 by \cite{utz1950unstable}. In \cite{reddy1968lifting} and \cite{keybob} it was shown that expansive $\mathbb{Z}$-systems admit symbolic extensions. From the work of Boyle and Downarowicz \cite{BDz2004} it follows that an expansive $\mathbb{Z}$-system admits a symbolic extension of the same entropy.
The notion of expansiveness for flows was introduced by Bowen and Walters \cite{bowen1972expansive}. Bowen and Walters \cite{bowen1972expansive} proved that expansiveness is invariant under topological conjugacy, that the topological entropy of an expansive flow is finite and that all fixed points of an expansive flow are isolated.
In addition they constructed a symbolic extension for expansive flows. They
asked whether this symbolic extension preserves entropy. More precisely, they made use of closed cross-sections to build a symbolic extension 
and wondered if one could choose carefully these closed cross-sections so that the associated symbolic extension has the same topological entropy as the original system. Burguet \cite{burguet2019symbolic} gave a positive answer to this question for $C^2$-expansive flows. In this paper, we give an affirmative answer for all expansive flows.

\begin{thm_b}
 Let $(X, \Phi)$ be an expansive flow. Then it has a strongly isomorphic symbolic extension.
\end{thm_b}

\subsection*{Structure of the paper} In Section \ref{sec:Preliminaries}, we recall basic notions related to discrete dynamical systems and topological flows. In Section \ref{sec:Establishing the small flow boundary property}, we prove that finite-dimensional topological flows without fixed points, having a countable number of periodic orbits, satisfy the small flow boundary property (Theorem A). In Section \ref{sec:Expansive flows have strongly isomorphic symbolic extensions}, we recall the definition of expansive flows and Bowen-Walters construction of symbolic extensions for expansive flows, and show that any expansive topological flow has a strongly isomorphic symbolic extension (Theorem \ref{thm:thm B}). This answers an open question of Bowen and Walters \cite{bowen1972expansive}.  In the Appendix (Section \ref{sec:Appendix}), we review the construction of a complete family of cross-sections for a topological flow without fixed points, following Bowen and Walters.

\subsection*{Acknowledgements} We are grateful to David Burguet who told us about the main problem of the paper and conveyed to us his strong conviction in the feasibility of a positive solution (see also \cite[Remark 2.3]{burguet2019symbolic}).

\section{Preliminaries}\label{sec:Preliminaries}

\subsection{Notation}\label{subsec:notation}
Let $(X,\dist)$ be a metric space. Let $x\in S\subset X$. Denote by $B_S(x,\epsilon)=\{y\in S|\, \dist(x,y)<\epsilon\}$ the open $\epsilon$-ball in $S$ around $x\in S$. If $S$ is clear from the context, it may be omitted from the notation. Let $C=\{A_1,\ldots A_n\}$ be a set of sets $A_i\subset X$, $i=1,\ldots,n$. We denote by $\bigcup C=\cup_{i=1}^n A_i\subset X$ and $\bigcap C=\cap_{i=1}^n A_i\subset X$. Let $S,V\subset X$. Denote by $\partial_S V$ and $\Int_S V$  respectively the \textbf{boundary} and \textbf{interior} of $V$ w.r.t.\ the subspace topology induced by $S$.  Let $S\subset X$ be a closed subset and $Q\subset S$ be a subset in $S$. Fix $\epsilon>0$. The \textbf{open}, respectively \textbf{closed $\epsilon$-tube around $Q$} is the set $\Theta^S_{\epsilon}(Q):=\{y\in S|\, \dist(y,\overline{Q})< \epsilon\}$ respectively $\overline{\Theta}^S_{\epsilon}(Q):=\{y\in S|\, \dist(y,\overline{Q})\leq \epsilon\}$. Let $O\subset S$  be an open subset in $S$. The \textbf{open}, respectively \textbf{closed internal $(-\epsilon)$-tube of  $O$ inside $S$} is the set $\Theta^S_{-\epsilon}(O):=\{y\in O|\, \dist(y,S\setminus O)> \epsilon\}$ respectively $\overline{\Theta}^S_{-\epsilon}(O):=\{y\in O|\, \dist(y,S\setminus O)\geq \epsilon\}$. If $S$ is clear from the context, it may be omitted from the notation. 

\subsection{Discrete dynamical systems and topological flows}
A pair $(X,T)$ is called a {\bf (discrete) dynamical system} if $X$ is a compact metric space and $T:X\to X$ is a homeomorphism. For $Y$, a second countable metrizable space we denote by $\dim(Y)$ the Lebesgue covering dimension of $Y$.\footnote{For second countable normal spaces (metrizable by the Urysohn metrization theorem), the Lebesgue covering dimension equals the small inductive dimension  (see \cite[\S 6.2]{fedorchuk1990fundamentals}).} We use the convention $\dim(Y)=-1$ iff $Y=\emptyset$. 
\begin{df}
	A \textbf{topological flow} is a pair $(X, \Phi)$ where $X$ is a compact metrizable space and $\Phi: X\times \R \to X$ is a continuous flow on $X$, that is, the map $\Phi$ is continuous, $\Phi(\cdot, 0)$ is the identity map on $X$ and $\Phi(\Phi(x,t), s)=\Phi(x, t+s)$ for all $t,s\in \R$ and $x\in X$. For $t\in \R$, we sometimes use the notation $\Phi_t(x)=\Phi(x, t)$ and notice that  $\Phi_t: X\rightarrow X$ is a homeomorphism. In addition for $L\subset \R$, $S\subset X$ we denote $\Phi_{L}(S)=\{\Phi_t(x)\,|\, t\in L,\, x\in S\}$. Throughout the text, we fix a compatible metric $d$ on $X$. Given a set $\emptyset \not= A\subset X$ and $x\in X$, we define $\dist(x,A)=\inf_{y\in A} \dist(x,y)$, as well as $\dist(x,\emptyset)=\infty$.
\end{df}
\begin{df}
	A point $x\in X$ is called a \textbf{fixed point} if $\Phi_t(x)=x$ for all $t$.
A point $x\in X$ is a \textbf{periodic point} with \textbf{period} $\tau>0$ if $\Phi_\tau(x)=x$ and $\Phi_t(x)\not=x$ for any $0<t<\tau$. In the latter case, the set $\{\Phi_t(x)\,|\,0\leq t \leq \tau\}$ is called the \textbf{periodic orbit} associated with $x$.
\end{df}
\begin{df}
	The flow $(X, \Phi)$ is said to be \textbf{aperiodic} if it has no periodic orbits, that is, the equation $\Phi_t(x)=x$ implies $t=0$.
\end{df}

\subsection{Cross-sections}

\begin{df}\label{df:cross-section}
A \textbf{cross-section} of \textbf{injectivity time} $\eta>0$ is a subset $S\subset X$ such that the restriction of $\Phi$ on $S\times [-\eta, \eta]$ is one-to-one. The cross-section is said to be \textbf{global} if there is $\xi>0$ such that $\Phi(S\times [-\xi, \xi])=X$. The set $\Phi_{[-\eta, \eta]}(S)$ is called the \textbf{$\eta$-cylinder} associated with $S$.
\end{df}
\begin{rem}
In \cite{burguet2019symbolic} a closed cross-section $S$ such that the flow map  $\Phi: (x, t) \rightarrow \Phi_t(x)$ is a surjective local homeomorphism from $S \times \R$ to $X$, is called a 
\textit{Poincaré cross-section} and it is shown this is strictly stronger than $S$ being a closed and global cross-section.   
\end{rem}
\begin{df}
Let $(X, \Phi)$ be a topological flow. A finite family $\SS=\{S_i \}_{i=1}^{N}$ of disjoint closed cross-sections each of injectivity time $\eta$ is said \textbf{complete} if $\bigcup_{i=1}^{N}\Phi_{[-\frac{\eta}{2},\frac{\eta}{2}]}(S_i)=X$. It follows that $\GG=\cup_{i=1}^{N}S_i$ is a closed global cross-section.
\end{df}

A trivial observation is if $x\in S$ and $\Phi_t x\in S$ for some $t\not=0$, where $S$ is a cross-section of injectivity time $\eta>0$, then $|t|> 2\eta$ as otherwise  $\Phi_{-\frac t 2}S\cap \Phi_{\frac t 2}S\neq\emptyset$. 

\subsection{Flow boundaries and interiors}\label{subsec:Calculating flow boundaries and interiors}

As we will see in the sequel, the cylinders associated with cross-sections will play a fundamental role in the analysis of topological flows. The following definition which originates in \cite[Definition 3]{bowen1972expansive} is based upon \cite[Lemma 2.1]{burguet2019symbolic}.

\begin{df}\label{def:Flow boundaries and interiors}
Let $U$ be a set contained in a closed cross-section  of injectivity time $\eta$.
The \textbf{flow interior} $\In(U)\subset U$ of $U$ is the unique set obeying:

$$\Phi_{(-\eta, \eta)}(\boldsymbol{{\rm Int}^{\Phi}(U)})={\rm Int}(\Phi_{[-\eta, \eta]}(U))$$

The \textbf{flow boundary} $\partial^{\Phi}U$ of $U$ is the unique set obeying:
$$\partial \Phi_{[-\eta, \eta]}(U)=\Phi_{-\eta}(U)\sqcup \Phi_{\eta}(U)\sqcup \Phi_{(-\eta, \eta)}( \boldsymbol{ \partial^{\Phi}U})$$

 One can show that for every $0<\gamma<\eta$:
$$
\In(U)=\text{Int}(\Phi_{[-\gamma, \gamma]}(U))\cap U,
$$ 
$$
\partial^{\Phi}U=\overline{U}\setminus \In(U).
$$
\end{df}
Note that $\partial^{\Phi}U$ is closed as it can be written as 
$\partial^{\Phi}U=\overline{U}\setminus \text{Int}(\Phi_{[-\gamma, \gamma]})$ 
for any $0<\gamma<\eta$ (recall that $\overline{U}$ is the  closure of $U$ w.r.t.\ the space $X$).
We remark that under certain circumstances the above notions coincide with the classical notions of interior and boundary under the induced topology on the cross-section. See Proposition \ref{prop:natural boundary}.

\subsection{Good cross-sections exist}\label{subsec:Good subsections exist}

Here are several important facts regarding cross-sections:

\begin{thm}(Whitney)\label{thm:Whitney} \cite[page 270]{whitney1933regular} 
Let $(X,\Phi)$ be a topological flow without fixed points, then for each $x\in X$ there is a closed cross-section $S$  such that $x\in \In S$.
\end{thm}

\begin{proof}(sketch)
Fix $x\in X$. For $y\in X$, define $$\theta(y)=\int_0^1 \dist(\Phi_s(y),x) ds,$$
where $d$ is the metric on $X$ compatible with the topology.
It is easy to see that for fixed $y\in X$ the function $t\mapsto \theta(\Phi_t(y))$ has a continuous derivative. We denote the derivative at $t=0$ by $\theta'(y)$. An easy calculation shows:
$$\theta'(y)=\dist(\Phi_1(y),x)-\dist(y,x).$$
Assume w.l.o.g.\ that $\Phi_1(x)\not= x$. Thus $\theta'(x)=\dist(\Phi_1(x),x)>0$. Using the inverse function theorem we find $\ell>0$ such that  $\theta(\Phi_{-\ell}(x))<\theta(x)<\theta(\Phi_{\ell}(x))$ and $\theta'(\Phi_{t}(x))>0$ for all $-2\ell\leq t\leq 2\ell$. Using the continuity of $\Phi$, $\theta$ and $\theta'$ we may find an open neighborhood $U$ of $x$ such that for all $y\in \overline{U}$, $\theta'(\Phi_{t}(y))>0$ for all $-2\ell\leq t\leq 2\ell$ and:
$$\theta(\Phi_{-\ell}(y))<\theta(x)<\theta(\Phi_{\ell}(y)).$$
Thus for every $y\in \overline{U}$ there is a unique $-\ell<t_y<\ell$ such that $\theta(\Phi_{t_y}(y))=\theta(x)$. It is not hard to show that the set 
$S=\{\Phi_{t_y}(y)\,|\, y\in \overline{U} \}$ is a closed cross-section such that $x\in S\cap U\subset \In S$. Indeed $\ell>0$ is clearly an injectivity time for $S$. Moreover $U\subset \Phi_{[-\ell, \ell]}(S)$, as if $t_y<0$ then $t_y+\ell>0$ and if $t_y>0$ then $t_y-\ell<0$. 
\end{proof}
Based on the previous theorem it is possible to prove: 
\begin{thm}(Bowen \& Walters)
A topological flow without fixed points admits a complete family of (closed) cross-sections.
\end{thm}

\begin{proof}
See  the proof of  Lemma \ref{lem:complete family appendix} where a stronger result is proven. 
\end{proof}

\begin{lem}\label{lem:dim_cross-section}
Let $(X, \Phi)$ be a topological flow with $\dim(X)=n\geq 1$. Let $S\subset X$ be a closed cross-section. Then $\dim(S)\leq n-1$. If in addition $S$ is a global cross-section then $\dim(S)= n-1$.
\end{lem}
\begin{proof}
As $S\times [-\eta, \eta]$ is homeomorphic to a subset of $X$, one has $\dim(S\times [-\eta, \eta])\leq n$. By a theorem of Hurewicz in \cite{hurewicz1935dimension}, as $S$ is compact, $\dim(S)\leq n-1$. Now assume in addition that  $S$ is a global cross-section. Let $\xi>0$ such that the natural continuous map $f: S\times [-\xi,\xi]\rightarrow X$ given by $(x,t)\mapsto \Phi(x,t)$ is surjective. Note  that $f$ is a closed map between two metric separable spaces and for every $x\in X$, $f^{-1}(x)$ is countable. Thus by \cite[Theorem 1.12.4]{engelking1995theory} 
$$\dim(X)\leq \dim(S\times[-\xi,\xi])+\sup_{x\in X}\dim(f^{-1}(x))\le  \dim(S)+1+0,
$$
where we used the product theorem for non-empty metric separable spaces $A$ and $B$ $\dim(A\times B)\leq \dim(A)+\dim(B)$ (\cite[Theorem 1.5.16]{engelking1995theory}).
Thus $\dim(S)\geq n-1$ as desired. 
\end{proof}

We remark that one can not replace the assumption ``global cross-section S” by ``a cross-section with $\In(S)\not=\emptyset$". A counter-example is that $X$ is a disjoint union of $X_1$ and $X_2$ with $\dim(X_1)>\dim(X_2)$, $\Phi$ is a continuous flow on $X$ without fixed points and $S\subset X_2$ is a cross-section with $\In(S)\not=\emptyset$. Then it is clear that $\dim(S)\le \dim(X_2)-1<\dim(X)-1$.

\subsection{Calculating with flow interiors and boundaries}\label{subsec:Calculating}

Let $S,V\subset X$. Recall that $\partial_S V$ and $\Int_S V$  denote respectively the boundary and interior of $V$ w.r.t.\ the subspace topology induced by $S$.

\begin{lem}[\cite{burguet2019symbolic}, Lemma 2.2]\label{lem:partial in induced topo}
	Let $S$ be a set contained in a closed cross-section. Let $U\subset S$. Then 
	$$
	\partial_S U \subset \Pa U \subset \partial_S U \cup \Pa S.
	$$

\end{lem}

\begin{ex}
	Consider a minimal rotation on the torus $(\mathbb{T}^2, \Phi)$. Identify $\mathbb{T}^2=[0,1]^2$ in the usual way. The closed set $S=[A,B]\times\{0\}$, where $0<A<B<1$ is a global cross-section. Let $A<C<B$ and $U=[C,B]\times\{0\}$ be a strict subset of $S$. Then $\partial_S U = \{C\}\times\{0\}$,  $ \Pa U =\{C, B\}\times\{0\} $ and $ \Pa S=\{A,B \}\times\{0\}$. This is an example for which the inclusion relations for sets in Lemma \ref{lem:partial in induced topo} are strict.
\end{ex}


\begin{lem}\label{lem:flow boundary formula}
	Let $S$ be a set contained in a closed cross-section. Let $U\subset S$. Then $\Pa U=\partial_S U \cup (\overline{U}\cap \Pa S)$.
\end{lem}
\begin{proof}
	By Lemma \ref{lem:partial in induced topo} and Definition \ref{def:Flow boundaries and interiors} it follows on the one hand that $\Pa U \subset (\partial_S U \cup \Pa S)\cap \overline{U}=\partial_S U \cup (\overline{U}\cap \Pa S)$. On the other hand, since $\overline{U}\cap \Pa S =\overline{U}\cap(\overline{S}\setminus \In S)\subset \overline{U}\setminus \In U=\Pa U$, we obtain by Lemma \ref{lem:partial in induced topo} that $\partial_S U \cup (\overline{U}\cap \Pa S)\subset \Pa U$.
\end{proof}

\begin{lem}\label{lem:In U=U}
	Let $(X,\Phi)$ be a topological flow. Let $U\subset S$ be two subsets of a closed cross-section. Then $\In U=U$ implies that $U$ is relatively open in $S$. If $U \cap \Pa S=\emptyset$, then the converse holds, i.e.\ $U$ relatively open in $S$ implies $\In U=U$.
\end{lem}
\begin{proof}
	Assume that $U=\In U$. Fix $0<\gamma<\eta$. From the definition of cross-section it is clear that $\Phi_{[-\gamma, \gamma]}(U)\cap S\subset U$. Therefore  $U=\In U=\text{Int}(\Phi_{[-\gamma, \gamma]}(U))\cap U=\text{Int}(\Phi_{[-\gamma, \gamma]}(U))\cap S$, implying that $U$ is relatively open in $S$.
	Now assume that $U$ is (relatively) open in $S$, i.e.\ $U\cap \partial_S U=\emptyset$. As it follows from Definition  \ref{def:Flow boundaries and interiors} that $\In U\subset U$ and $\In U\cup \Pa U=\overline{U}$, it is enough to prove   $\Pa U\cap U=\emptyset$. By Lemma \ref{lem:flow boundary formula}, $\Pa U\cap U=(\partial_S U\cap U) \cup (U\cap \Pa S)$. This expression equals the empty set as by assumption $U \cap \Pa S=\emptyset$. 
\end{proof}

 Given a cross-section $S$ and $x\in \In S$, we clearly have $\delta:=\dist(x,\Pa S)>0$ as $\Pa S$ is closed. Therefore it is easy to see that $U=B(x, \delta/2)\cap S$ is a relatively open set in $S$ with $x\in U\subset \In S$ and $\overline{U} \cap \Pa S=\emptyset$. This explains the usefulness of the criterion in the following proposition.
Indeed this proposition is key for calculating flow boundaries and interiors in the theory we develop below.
\begin{prop}\label{prop:natural boundary}
Let $(X,\Phi)$ be a topological flow. Let $U$ be a subset of a closed cross-section $S$ such that 
$\overline{U} \cap \Pa S=\emptyset$.
Then:
$$\Pa U=\partial_S U\,\, \mathrm{ and }\,\, \In U=\Int_S U.$$
\end{prop}
\begin{proof}
 By Lemma \ref{lem:flow boundary formula},   $\Pa U=\partial_S U$. Thus $\Int_S U=\overline{U}\setminus \partial_S U=\overline{U}\setminus \Pa U=\In U$.
\end{proof}

\begin{rem}
A case where the previous remark can be used is when $S$ is a cross-section, $U,V \subset S$ with $\In U=U$ and $\In V=V$ and $\Pa U\subset V$. Indeed $$\overline{U}=\In U\cup\Pa U\subset \In U\cup V=\In U\cup \In V \subset\In S $$
Thus $\overline{U} \cap \Pa S=\emptyset$.
\end{rem}

\subsection{The small flow  boundary property}
First, we recall the small boundary property for discrete dynamical system.
\begin{df}(\cite{SW,LW})
Let $(X,T)$ a discrete dynamical system.  A subset $A\subset X$ has a {\bf small boundary} if $\mu(\partial A)=0$ for every $T$-invariant measure $\mu$. The system $(X,T)$ is said to have the {\bf small boundary property} if there
is a basis for the topology of $X$ consisting of open sets with small boundary.
\end{df}
Let $(X, \Phi)$ be a topological flow.
\begin{df}
 A Borel subset of $X$ is called a \textbf{null set} if it has zero measure w.r.t.\ any $\Phi$-invariant Borel probability measure. A Borel subset is said to be a \textbf{full set} when its complement is a null set. 
 \end{df}
 
\begin{df}
	A closed cross-section $S$  of injectivity time $\eta$ has a \textbf{small flow boundary} if $\Phi_{[-\eta,\eta]}(\partial^{\Phi}(S))$ is a null set.
\end{df}

For a cross-section $A\subset X$, we define the \textbf{counting orbit capacity} of $A$ by
$$
\cocap(A):=\lim_{T\to \infty} \frac{1}{T}\sup_{x\in X} \sharp\{0\leq t<T:\Phi_t(x)\in A \}.
$$

The limit exists and is finite as $\sup_{x\in X} \sharp\{0\leq t<T:\Phi_t(x)\in A \}$ is subadditive. 

\begin{lem}[\cite{burguet2019symbolic}, Lemma 2.10]\label{lem:small flow boundary}
	Let $(X, \Phi)$ be a topological flow. Suppose that $S$ is a closed  cross-section  of injectivity time $\eta>0$. Then the following are equivalent.
	\begin{enumerate}

		\item  $S$ has a small flow boundary;
		\item $\cocap(\partial^{\Phi}S)=0$.
	\end{enumerate}
\end{lem}

\begin{df}
Let $(X,\Phi)$ be a topological flow. The flow $(X, \Phi)$ is said to have the \textbf{small flow boundary property} if for any $x\in X$ and any closed cross-section $S'$ with $x\in \In(S')$, there exists a closed subset $S\subset S'$ such that $x\in \In(S)$ and $S$ has a small flow boundary.
\end{df}
\begin{rem}
    Note that if a topological flow has the small flow boundary property than it has no fixed points as it is required that every $x\in X$ belongs to a cross-section.	
\end{rem}
\subsection{Suspension flows and symbolic extensions}

Let $(Z, \rho)$ be a compact metric space and $T:Z\to Z$ a homeomorphism. Let $f: Z\to \mathbb{R}_{>0}$ be a continuous map. Let $
F_{f}=\{(x,t): 0\le t\le f(x), x\in Z \}\subset Z\times \R$. Let  $R_f$ be the closed equivalence relation on $F_{f}$ induced by $\{((x,f(x)),(Tx,0))|\, x\in Z\}\subset   F_{f}\times F_{f}$.

The \textbf{suspension flow} of $T$ under $f$ is the flow $\Phi$ on the space
$$
Z_{f}:=F_f/ R_f
$$
induced by the time translation $T_t$ on $Z\times \mathbb{R}$ defined by $T_t(x,s)=(x, t+s)$. A suspension flow over a zero-dimensional $\mathbb{Z}$-topological dynamical system is called a \textbf{zero dimensional suspension flow} and a topological extension by a zero-dimensional suspension flow is said to be a \textbf{zero-dimensional extension}. Similarly a suspension flow over a symbolic $\mathbb{Z}$-topological dynamical system (a.k.a $\mathbb{Z}$-\textbf{subshift}) is called a \textbf{ symbolic suspension flow} and a topological extension by a symbolic suspension flow is said to be a \textbf{symbolic extension}.
 \begin{df}\label{df:principal}
Let $(X, \Phi)$ and $(Y, \Psi)$ be two topological flows. Suppose that $\pi: Y\to X$ is a topological extension from $(X, \Phi)$ to $(Y, \Psi)$. A topological extension is said to be (see \cite[\S 2.3]{burguet2019symbolic})
\begin{itemize}
   \item \textbf{entropy-preserving} when it preserves topological  entropy i.e.\ $\htop(X, \Phi)=\htop(Y, \Psi)$.
    \item \textbf{principal} when it preserves the entropy of invariant measures, i.e.\ $\h(\mu,\Psi)=\h(\pi_* \mu,\Phi)$ for all $\Psi$-invariant measures $\mu$,
\item \textbf{isomorphic} when $\pi: (Y, \Psi, \mu) \to (X, \Phi, \pi_* \mu)$ is a measure theoretical isomorphism for all $\Psi$-invariant measures $\mu$,
\item \textbf{strongly isomorphic} when there is a full set $E$ of $X$ such that the restriction of $\pi$ to $\pi^{-1}E$ is one-to-one.
\end{itemize}
\end{df}

\begin{rem}
 
Clearly, strongly isomorphic $\Longrightarrow$ isomorphic $\Longrightarrow$ principal$\Longrightarrow$ entropy-preserving. It is easy to give an example of an extension which is  entropy-preserving but not principal: Let $\pi: (X,T)\rightarrow(Y,S)$ such that $\htop(T)>\htop(S)$. Then let $f: X  \sqcup X \to X  \sqcup Y$ by $f(x)=x$ if $x$ in first $X$ and $f(x)=\pi(x)$ if $x$ in second $X$. An example of an extension which is principal but not isomorphic is given by \cite[Theorem 4.7]{burguet2019uniform}. 
Downarowicz and Glasner constructed
an minimal system which is an isomorphic but not almost 1–1 extension of its maximal equicontinuous factor. If this extension were  strongly isomorphic then it would trivially have at least one singleton fiber implying the extension is almost 1–1. Thus their construction also  gives an example of an extension which is isomorphic but not strongly isomorphic (see \cite[Remark 3.2(d)]{downarowicz2016isomorphic}).
Note that the $1$-suspensions over the examples above give the analogous  examples for flows. Indeed by \cite[p. 4328]{burguet2019symbolic}, there is an affine homeomorphism between the simplices of invariant measures $\Theta:\mathcal{M}(X,T)\rightarrow \mathcal{M}(Z_1(X),\Phi)$ given by $\mu \mapsto \mu\times \lambda$, where $\lambda$ is the Lebesgue measure on the interval $[0,1]$. 
\end{rem}

\subsection{Some facts from dimension theory}\label{subsec:facts_dimension_theory}

The following results in dimension theory for a separable metric spaces  will be used in the proof of our main result. Recall that an $F_\sigma$ set is a countable union of closed sets.  
\begin{thm}(\cite[Proposition 1.2.12]{engelking1995theory})\label{thm:R1}
Let $E$ be a zero-dimensional subset in a separable metric space $M$. Then for every $x\in M$ and every open neighborhood $U$ of $x$ there is an open set $U'\subset U$ with $x\in U'$ and $\partial U' \cap E=\emptyset$
\end{thm}

\begin{thm}(\cite[Chapter III, Theorem III 2]{HW41})\label{thm:R3}
Let $(B_i)_{i\in \mathbb{N}}$ be a countable collection of closed sets in a separable metric space satisfying $\dim B_i\le k$ for all $i$. Then $\dim \bigcup_i B_i\le k$. 
\end{thm}
The following corollary is obvious.
\begin{cor}\label{cor:sum_theorem_for_F_sigma}
Let $(B_i)_{i\in \mathbb{N}}$ be a countable collection of $F_\sigma$ sets in a separable metric space satisfying $\dim B_i\le k$ for all $i$. Then $\dim \bigcup_i B_i\le k$. 
\end{cor}
\begin{lem}\label{lem:clo_int_open_is_F_sigma}
Let $M$ be a separable metric space. Let $C$ be closed and $U$ open then $U\cap C$ is $F_\sigma$. 
\end{lem}
\begin{proof}
    Note $M$ is second-countable and write $U$ as a countable union of closed balls. 
\end{proof}

The following theorem is well known.
\begin{thm}(\cite[Theorem 1.5.7]{engelking1995theory})\label{thm:R22}
Let $M$ be a separable metric space. If $-1<\dim(M)<\infty$ there is a zero-dimensional subset $E$ of $M$  such that $\dim(M\setminus E)=\dim(M)-1$.
\end{thm}

We will need the following strengthening of Theorem \ref{thm:R22}.
\begin{thm}\label{thm:R2}
Let $M$ be a separable metric space. If $-1<\dim(M)<\infty$ there is a zero-dimensional $F_\sigma$ subset $E$ of $M$  such that $\dim(M\setminus E)=\dim(M)-1$.
\end{thm}
\begin{proof}
    Denote $n=\dim(M)$. If $n=0$, set $E=M$. If $n=1$, let $\{B_i\}_{i=1}^\infty$ be a countable basis of open sets of $M$ so that $\dim(\partial B_i)=0$ for all $i$ (\cite[Theorem 1.1.6]{engelking1995theory}). Let $E=\bigcup_i\partial B_i$. By Theorem \ref{thm:R3} $E$ is a zero-dimensional $F_\sigma$ set. Note $\{B_i\cap (M\setminus E)\}_{i=1}^\infty$ is a  basis of $M\setminus E$ with $\partial (B_i\cap (M\setminus E))=\emptyset$. Thus $\dim (M\setminus E)=0$ as desired. Assume the claim has been established for $n-1\geq 0$. Let $\{B_i\}_{i=1}^\infty$ be a countable basis of $M$ so that $\dim(\partial B_i)\leq n-1$ for all $i$. 
    Using the inductive assumption let $F_i\subset \partial B_i $ be a zero-dimensional $F_\sigma$ set so that $\dim(\partial B_i\setminus F_i)\leq n-2$. Let $E=\cup_i F_i$. By Theorem \ref{thm:R3}, $E$ is a zero-dimensional $F_\sigma$ set. Note $\{B_i\cap (M\setminus E)\}_{i=1}^\infty$ is a  basis of $M\setminus E$ with $\partial (B_i\cap (M\setminus E))\subset \partial B_i\setminus F_i$. Thus $\dim (M\setminus E)\leq n-1$. By \cite[Chapter III, 2 B]{HW41}, $\dim (E\cup (M\setminus E))\leq 1+\dim (E)+ \dim(M\setminus E)$ and thus $\dim (M\setminus E)= n-1$.
\end{proof}

As an illustration of Theorem \ref{thm:R2} consider $M=[0,1]^2$ and note that $E=M\cap (\mathbb{Q}\times\mathbb{Q})$ is zero-dimensional and $M\setminus E$ is one-dimensional. Indeed it is easy to see that balls centered at rational coordinates with a rational radius form a base with zero-dimensional boundary in $M\setminus E$. As an illustration of Theorem \ref{thm:R1} note that  balls centered at rational coordinates with a transcendental radius do not intersect $E$.
\begin{prop}\label{prop:avoiding zero_dim}
 Let $M$ be a metric compact space and $E$ a zero-dimensional subset of $M$. Let $C$ be a closed set and $U$ an open subset such that $C\subset U\subset M$.  Then there is an open set $U'$ with $C \subset U'\subset \overline{U'}\subset U$ and $\partial U' \cap E=\emptyset$.
\end{prop}
\begin{proof}
First fix an open set $W$ so $C \subset W\subset \overline{W}\subset U$. Using Theorem \ref{thm:R1}, for each $x\in C$, choose an open set $x\in U_x\subset W$ with $\partial U_x \cap E=\emptyset$.  Let $U_{x_1},\ldots, U_{x_n}$ be a  cover of $C$. It is easy to see that $U'=\bigcup_{i=1}^n U_{x_i}\subset W$ has the required properties.
\end{proof}

\section{Establishing the small flow boundary property }\label{sec:Establishing the small flow boundary property}
\subsection{Notation and assumptions in force in Section 3}\label{subsec:notation_3}
Let $(X, \Phi)$ be a topological flow without fixed points. Through this section, we always suppose that  $\dim(X)=d+1$ for $d\ge 0$.
By Lemma \ref{lem:complete family appendix} there exists an $\eta>0$  and $0<\alpha<\eta$ such that $(X, \Phi)$ has two complete families of cross-sections $\SS=\{S_i \}_{i=1}^{N}$ of (closed disjoint) cross-sections of injectivity time $\eta_{\SS}$ and  $\SS'=\{S_i' \}_{i=1}^{N}$  of injectivity time $\eta_{\SS}'$ such that for all $1\le i\le N$,
\begin{itemize}
	\item $S_i=\overline{S_i}\subset \In(S_i')$; 
	\item $\eta=\eta_{\SS}=\eta_{\SS'}$;
	\item $\max \{  \diam (S_i')\}\le \alpha$;
	\item $\Phi_{[0,\alpha]}\mathcal{G}=\Phi_{[-\alpha, 0]}\mathcal{G}=X$, where $\mathcal{G}=\cup_{S\in \mathcal{S}} S$.
\end{itemize}
\noindent
Starting with Definition \ref{def:n-general} in the sequel we fix two cross-sections $S=S_i$ and $S'=S_i'$ in the complete families $\mathcal{S}$ and $\mathcal{S}'$ for some $1\le i\le N$. 

\subsection{Local homeomorphisms between cross-sections }\label{sec:notations}

\begin{df}
Let $1\le i,j\le N$. Denote by $t_{i,j}$ the first positive \textbf{hitting time} from $S_i$ to $S_j$, that is, for $x\in S_i$, 
$$
t_{i,j}(x)=\min\{t>0: \Phi_t(x)\in S_j \},
$$
where if the argument set is empty we put $t_{i,j}(x)=\infty$. The functions $t_{i,j}'$ are defined similarly w.r.t.\ $S_i'$ and $S_j'$. Note $t_{i,j}'(x)\leq t_{i,j}(x)$ for all $x\in S_i$. Define
$$D_{i,j}=\{x\in S_i: t_{i,j}(x)\le \eta \}$$
and
$$D_{i,j}'=\{x\in S_i': t_{i,j}'(x)<\infty\}.$$ 
Note $D_{i,i}=\emptyset$ for all $i$. Let $T_{i,j}$ be the first positive \textbf{hitting map} from $D_{i,j}$ to $S_j$, that is, for $x\in D_{i,j}$, 
$$
T_{i,j}(x)=\Phi_{t_{i,j}(x)}(x),
$$
Similarly, we define $T_{i,j}'$ from $D_{i,j}'$ to $S_j'$. 

\end{df}

\begin{df}\label{def:i_th return map}
Define the \textbf{first return-time} map\\ $t^{(1)}_\mathcal{G}(x)=t_\mathcal{G}~:~ \mathcal{G}~=~\cup_{S\in \mathcal{S}} S\to \R_+$ by $$t^{(1)}_\mathcal{G}(x)=t_\mathcal{G}(x):=\min \{t>0|\, \Phi_t(x)\in \mathcal{G}\}.$$
As $\SS=\{S_i \}_{i=1}^{N}$ is a complete family of closed disjoint cross-sections, there is $\gamma>0$ so that $S=\bigcup_{i=1}^{N} S_i$ is a global cross-section with injectivity time $\gamma$. Thus the argument set in the definition above is never empty.

Inductively for $i\geq 1$ define the \textbf{$(i+1)$-st return time} map \\ $t^{(i+1)}_\mathcal{G}(x)~:~ \mathcal{G}~=~\cup_{S\in \mathcal{S}} S\to \R_+$ by $$t^{(i+1)}_\mathcal{G}(x)=t^{(i)}_\mathcal{G}(t_\mathcal{G}(x))+t_\mathcal{G}(x).$$

Define $t^{(0)}_\mathcal{G}(x)=0$ and $t^{(-k)}_\mathcal{G}(x)$ for $k\in \N$, similarly to the above.

\end{df}

\begin{df}
For $1\leq i,j\leq N$, define:

$$F_{i,j}:=\{x\in D_{i,j}|\, t_{i,j}(x)=t_\mathcal{G}(x)\}. $$
That is $F_{i,j}$ is the set of $x\in S_i$ such that the first member of  $\SS$ it hits after leaving $S_i$ is $S_j$. 
\end{df}

\begin{lem}\label{lem:F_ij_bounded_below}
If $x\in F_{i,j}$, then $t_{i,j}(x)>\gamma$. 
\end{lem}
\begin{proof}
 Note  $t_\mathcal{G}$ is bounded  from below by $2\gamma$.
\end{proof}

\begin{lem}\label{lem:F_covers}
For all $1\leq i\leq N$, it holds that  $S_i=\cup_{j=1}^N F_{i,j}$.
\end{lem}
\begin{proof}
Let $x\in S_i$. As $\Phi_{[-\alpha, 0]}\mathcal{G}=X$, there exists some $j$, $t\in [0,\alpha]$ and $y\in S_j$ so that  $\Phi_{-t}y=x$, i.e., $\Phi_{t}x=y\in S_j$ . Thus $t_{i,j}(x)\leq \alpha<\eta$ and $x\in D_{i,j}$. Let $K=\{1\leq k\leq N|\, x\in  D_{i,k}\}$. Note $j\in K$. Clearly there is $k\in K$ so that $t_{i,k}=\min_{s\in K} t_{i,s}$. Conclude $x\in  F_{i,k}$.
\end{proof}

\begin{lem}\label{lem:identify_t_ij}
Let $1\leq i,j\leq N$.
\begin{enumerate}

    \item 
If $x\in S_i'$, $y\in S_j'$ and $\Phi_r(x)=y$ for some $0\leq r\leq 2\eta$, then $t_{i,j}'(x)=r$.
\item
If $x\in D_{i,j}$, then $t_{i,j}'(x)=t_{i,j}(x)$.
\end{enumerate}
\end{lem}
\begin{proof}
  For (1) assume for a contradiction that $t_{i,j}'(x)<r$. It follows that $\Phi(y, r-t_{i,j}'(x))\in S_j'$ and $|t_{i,j}'(x)-r|\leq 2\eta$. Since $\Phi$ is injective on $S_j'\times [-\eta, \eta]$, this is a contradiction. 
  
  For (2) note that by definition $t_{i,j}(x)\leq \eta$ and  $\Phi_{t_{i,j}(x)}(x)\in S_j'$, thus by (1), $t_{i,j}'(x)=t_{i,j}(x)$.
\end{proof}

\begin{prop}\label{Prop:T_homeomorphism}
	 Let $x\in D_{i,j}$. Then there is an open neighborhood $V$ of $x$ in $S_i'$ with $\overline{V}\subset \In S_i'$ such that $t_{i,j}'$ is continuous on $V$, $t_{i,j}'(x)\leq 2\eta$ for all $x\in V$ and $T_{i,j}'|_{V}:V\rightarrow S_j'$ is an open map and $T_{i,j}'|_{V}:V\rightarrow T_{i,j}'(V)$ is a homeomorphism.
\end{prop}
\begin{proof}
	Let $0<\epsilon<\eta$.
	Since $T_{i,j}(x)\in S_j\in \In S_j'$, there is a $\kappa>0$ such that the open ball $B(T_{i,j}(x), \kappa)$ is contained in $\Phi_{(-\epsilon, \epsilon)}(\In S_j')$. Since $\Phi$ is continuous, there is an $\delta>0$ such that for any $y\in X$ with $\dist(x,y)<\delta$, we have that $\Phi(y, t_{i,j}(x))\in B(T_{i,j}(x), \kappa)\subset \Phi_{(-\epsilon, \epsilon)}(\In S_j')$. Let $V_{\delta}=\{y\in S_i': \dist(x,y)<\delta\}$ which is an open neighborhood of $x$ in $S_i'$. As $x\in S_i$, $\dist(x, \Pa S_i')>0$, thus taking $\delta$ small enough, one may assume w.l.o.g.\ that $\overline{V_{\delta}}\subset \In S_i'$. It follows by Lemma \ref{lem:identify_t_ij} that
	\begin{equation}\label{eq:two_sided_ineq}
	t_{i,j}(x)-\epsilon <t_{i,j}'(y)<t_{i,j}(x)+\epsilon,   
	\end{equation}
	 whenever $y\in V_{\delta}$. In particular $t_{i,j}'(y)< 2\eta$ for all $y\in V_\delta$  and $T_{i,j}'(V)\subset \In S_j'$. We claim that after taking $\delta$ small enough, the map $t_{i,j}'$ is continuous on $V=V_{\delta}$. If not, then there exists a sequence $\{y_n\}_{n\in \N}\subset V$ such that $\lim\limits_{n\to \infty}y_n=x$  but $t_x:=\lim\limits_{n\to \infty}t_{i,j}'(y_n)\not=t_{i,j}(x)$. As by definition $0<t_{i,j}(x)\leq \eta$ and by choice $\epsilon<\eta$, one has that $t_x\le t_{i,j}(x)+\epsilon<2\eta$ and $\Phi(x, t_x)\in S_j'$. It follows that $\Phi(T_{i,j}(x), t_x-t_{i,j}(x))\in S_j'$ and $|t_x-t_{i,j}(x)|<2\eta$. Since $\Phi$ is injective on $S_j'\times [-\eta, \eta]$, this is a contradiction. Let $z=T_{i,j}'(x')\in \In S_j'$ for some $x'\in V$. Let $W_{\rho}=\{y\in S_j': \dist(z,y)<\rho\}$ be an open neighborhood of $z$ in $S_j'$. Using Equation \eqref{eq:two_sided_ineq} with $y=x'$, for $\rho>0$ small enough, there is $0<\xi<\eta$ such that  $\Phi_{-t_{i,j}(x)}(\Phi_{(-\rho, \rho)}W_{\rho})$ is an open set in $\Phi_{(-\xi, \xi)}(V)$. Thus for each $w\in W_{\rho}$, there are unique $v\in V$, $r\in (-\xi, \xi)$, such that $\Phi_{t_{i,j}(x)-r}(v)=w$. By Lemma \ref{lem:identify_t_ij}, one must have $t_{i,j}'(v)=t_{i,j}(x)-r$ and $T_{i,j}'(v)=w$. This implies $W_\rho\subset T_{i,j}'(V)$ which implies $T_{i,j}'|_{V}:V\rightarrow S_j'$ is an open map. In addition, as $t_{i,j}'$ is continuous on $V$ so is $T_{i,j}'|_{V}$. Thus in order to establish that $T_{i,j}'|_{V}:V\rightarrow T_{i,j}'(V)$ is a homeomorphism, it is enough to show that $T_{i,j}'|_{V}$ is injective. Indeed if it holds for $z_1,z_2\in V$, that $T_{i,j}'(z_1)=T_{i,j}'(z_2)$, that is $\Phi_{t_{i,j}'(z_1)}(z_1)=\Phi_{t_{i,j}'(z_2)}(z_2)$, then $z_1=\Phi_{t_{i,j}'(z_2)-t_{i,j}'(z_1)}(z_2)$ which implies  $|t_{i,j}'(z_2)-t_{i,j}'(z_1)|>2\eta$, contradicting Equation \eqref{eq:two_sided_ineq}. 
		Q.E.D.
\end{proof}

\subsection{Return time sets and cross-section names}\label{sec:return_names}

\begin{prop}\label{prop:C_ij}
For every $1\leq i,j\leq N$ there is a finite collection of open sets $\CC_{i,j}$ in $S_i'$ with the following properties:
\begin{enumerate}
    \item $F_{i,j}\subset \bigcup \CC_{i,j}\subset \overline{ \bigcup \CC_{i,j}}\subset \In S_i'$.
    \item For all $V\in \CC_{i,j}$,  $t_{i,j}'$ is continuous on $V$ and $t_{i,j}'(x)\leq 2\eta$ for all $x\in V$.
    \item For all $V\in \CC_{i,j}$, $(T_{i,j}')_{|V}:V\rightarrow T_{i,j}'(V)$ is a homeomorphism and $T_{i,j}'(V)$ is open in $S_j'$.
   \end{enumerate}
\end{prop}
\begin{proof}
Fix $1\le i,j\le N$.  Let $x\in \overline{F_{i,j}}$. Let $\{x_q\}_{q=1}^{\infty}\subset F_{i,j}\subset D_{i,j}$ so that $x_q\rightarrow_{q\rightarrow \infty} x$. As $t_{i,j}(x_q)\leq \eta$, we may assume w.l.o.g.\ $t_{i,j}(x_q)$ converges to some $t\leq \eta$. Clearly $\Phi_t(x)\in S_j$ and therefore $x\in D_{i,j}$. Conclude $\overline{F_{i,j}}\subset D_{i,j}\subset S_{i}\subset \In S'_i$. Using Proposition \ref{Prop:T_homeomorphism} cover $\overline{F_{i,j}}$ by a finite collection of open sets $\CC_{i,j}$ such that  $\bigcup\CC_{i,j}\subset \overline{ \bigcup \CC_{i,j}}\subset \In S_i'$. with properties  (2) and (3).      
\end{proof}
For each $C\in \CC_{i,j}$, we define $T_C:C\to S_j'$ by $x\mapsto T_{i,j}'(x)$. For $n\in \N\cup \{0\}$, $\textbf{i}=(i_k)_{0\le k\le n}\in\{1,2,\dots N\}^{n+1}$ and $C^n=(C_0, C_1, \dots, C_{n-1}, C_n)\in \CC(\textbf{i}):=\prod_{k=0}^{n-1}\CC_{i_k, i_{k+1}}\times \{\CC_{i_n,j} \}_{1\le j\le N}$, we define 
$$
T_{C^n}^k=
\begin{cases}
T_{C_{k-1}}\circ \cdots \circ T_{C_0}~&\text{for}~1\le k\le n;\\
\text{Id}_{C_0} &\text{for}~k=0;
\end{cases}
$$
on the set
$$
Z_{C^n}:=\{x\in C_0: T_{C_{k-1}}\circ \cdots \circ T_{C_0}(x)\in C_k~\text{for}~1\le k\le n  \}.
$$
Let $x\in Z_{C^n}$. The set $Z_{C^n}$ is called an  $n$-\textbf{cross-section name} of $x$. Note $x$  may have more than one  $n$-cross-section name for a given $n \in \N$.

Let $V\subset X$, not necessarily contained in the image of $T_{C^n}$. Following standard notation, we denote 
\begin{equation}\label{eq:T_C_n}
T_{C^n}^{-k}(V):=\{ x\in Z_{C^n}: T_{C^n}^k(x)\in V\}=(T_{C^n}^{k})^{-1}(V\cap T_{C^n}^k(Z_{C^n})).
\end{equation} 

Let $n\in \N\cup \{0\}$. Denote $$\CC(n)= \cup_{\textbf{i}\in \{1,2,\dots N\}^{n+1}} \CC(\textbf{i}).$$

$$\CC_i=\cup_{1\leq j\leq N} \CC_{i,j}.$$

\begin{df}\label{def:I_C_n}
    
Let $n\in \N\cup \{0\}$. Define the \textbf{return time set}:
$$
I_{C^n}=\{0\le k\le n: i_{k}=i_0 \}.
$$
It follows that the image of $T_{C^n}^k$ is contained in $S_{i_0}'$ for $k\in I_{C^n}$. Note $0\in I_{C^n}$.  
\end{df}

\begin{rem}
 Note that $\CC(0)=\bigcup_{1\le i,j\le N} \CC_{i,j}$.
Note that for $C^0=(C_0)\in \CC(0)$ it holds $I_{C^0}=\{0\}$, $Z_{C^0}=C_0$ and  $T_{C^0}^0=\text{Id}_{C_0}$.

\end{rem}

\begin{lem}\label{lem:F_sigma_invariance}
Let $n\in \N$, $C^n\in \CC(n)$ and $1\leq k\leq n$. Then the following holds:
\begin{enumerate}
    \item 
$Z_{C^n}$ and $T_{C^n}^{k}(Z_{C^n})$ are open;  $T_{C^n}^{k}$ is a homeomorphism on $Z_{C^n}$.
\item
Let  $V$ be $F_\sigma$. Then $T_{C^n}^{-k}(V)$ is $F_\sigma$.
\end{enumerate}
\end{lem}
\begin{proof}
The first claim is best understood by considering the first cases. Indeed note $T_{C^n}^1=T_{C_0}:C_0\rightarrow T_{i_0,i_1}'(C_0)$ and $T_{C^n}^2:T_{C_0}^{-1}(T_{C_0}(C_0)\cap C_1)\rightarrow T_{C_1}(T_{C_0}(C_0)\cap C_1)$, both have a  domain which is an an open set in $S_{i_0}'$ and a range which is an open set in $S_{i_1}'$ and $S_{i_2}'$ respectively by Proposition \ref{Prop:T_homeomorphism}. Now assume we have proven that the domain of $T_{C^n}^k$ is an open set $D_k$ in $S_{i_0}'$ and that the range of $T_{C^n}^k$ is an open set $R_k$ in $S_{i_k}'$. It is easy to see that the domain of $T_{C^n}^{k+1}$  is an open set in $D_k$  and therefore in $S_{i_0}'$,  $(T_{C_n}^k)^{-1}(R_k\cap C_k)$ and that the range of $T_{C^n}^{k+1}$ is an open set in $S_{i_{k+1}}'$,  $T_{C_k}(R_k\cap C_k)$. Note that $Z_{C^n}$ equals the domain of $T_{C^n}^{n}$ which we have seen to be open. Using  Proposition \ref{Prop:T_homeomorphism}(3), we have that $T_{C^n}^{k}(Z_{C^n})$ is open. By the above $T_{C^n}^{k}$ is a homeomorphism on an open domain which contains $Z_{C^n}$ and therefore is  a homeomorphism on $Z_{C^n}$. 

For Claim $(2)$, note $T_{C^n}^{-k}(V)=(T_{C^n}^{k})^{-1}(V\cap T_{C^n}(Z_{C^n}))$ and $T_{C^n}(Z_{C^n})$ is open by (1). Thus $V\cap T_{C^n}(Z_{C^n})$ is $F_\sigma$ and on this set $T_{C^n}^{-k}$
is a homeomorphism.
\end{proof}

Using Proposition \ref{prop:C_ij} let $\delta_0>0$ small enough so that $2\delta_0$ is a Lebesgue number for each of the open covers  $\CC_{i,j}$ (of $F_{i,j}$). Thus for all $1\leq i,j\leq N$,  
 \begin{equation}\label{eq:delta_approx}
F_{i,j}\subset \bigcup_{V\in \CC_{i,j}}\overline{\Theta}^{S_i'}_{-\delta_0}(V).     
 \end{equation}
 From now on in this section we abbreviate $\overline{\Theta}_{\cdot}(\cdot)=\overline{\Theta}^{S_1'}_{\cdot}(\cdot)$.
\begin{df}\label{def:delta_0}
We define
 $$
Z_{C^n}^{\delta_0}:=\{x\in \overline{\Theta}_{-\delta_0}(C_0): T_{C_{k-1}}\circ \cdots \circ T_{C_0}(x)\in \overline{\Theta}_{-\delta_0}(C_k)~\text{for}~1\le k\le n  \}.
$$
This is clearly a closed set.
\end{df}

Recall that $\dim(X) = d + 1$. The following lemma will be important in the sequel:
\begin{lem}\label{lem:internal_return}
Let $x\in S_1$ so that there are $0=t_0<t_1<t_2<\ldots<t_{d}$ with $\Phi_{t_k}x\in S_1$ for $k=0,\ldots, d$.  Then there exists  $C^n=(C_0, C_1, \dots, C_{n-1}, C_n)\in \CC(n)$ for some $n\in\N$, and $j_0=0<j_1<\ldots <j_d$ such that $x\in Z_{C^n}^{\delta_0}$ and  $T_{C^n}^{j_i}x=\Phi_{t_i}x$ 
for $i=0,\ldots, d$. 
\end{lem}

\begin{proof}
By Lemma \ref{lem:F_covers}, $x\in F_{1,j_1}$ for some $j_1$. By Equation \eqref{eq:delta_approx}, we may find $C_0\in \CC_{1,j_1}$ so that $x\in \overline{\Theta}_{-\delta_0}(C_0)$. A priori $T_{C_0}(x)\in S_{j_1}'$, however by Lemma \ref{lem:identify_t_ij}(2), $T_{C_0}(x)\in S_{j_1}$. Thus by Lemma \ref{lem:F_covers}, $T_{C_0}(x)\in F_{j_1,j_2}$ for some $j_2$.  Repeating the argument so that we may find $C_1\in \CC_{j_1,j_2}$ so that
$T_{C_0}(x)\in \overline{\Theta}_{-\delta_0}(C_1)$. Continuing inductively we construct a sequence $C^n:=(C_0,C_1,\ldots C_n)\in\CC(n)$ with the properties:   
\begin{enumerate}
    \item $x\in Z_{C^n}^{\delta_0}$.
    \item $T_{C^n}^k(x)=\Phi_{t_\mathcal{G}(T_{C^n}^{k-1}(x))}(T_{C^n}^{k-1}(x))$, for $k=1,\ldots,n$.
    \item $n\gamma\geq t_d.$
\end{enumerate}

Conditions (2) and (3) guarantee that there exist $j_0=0<j_1<\ldots <j_d$ such that  $T_{C^n}^{j_i}x=\Phi_{t_i}x$ 
for $i=0,\ldots, d$.
 
\end{proof}

\subsection{Statement of the Main Theorem}
The main result is as follows.

\begin{thm}[=Theorem A]\label{main thm}
	Let $X$ be a finite-dimensional space.
    Let $\Phi$ be a topological flow on $X$ without fixed points, having a countable number of periodic orbits. Then  $(X, \Phi)$ has the small flow boundary  property.
\end{thm}


We follow the strategy of \cite{L95}. Lindenstrauss proved that for any discrete finite dimensional dynamical system $(X,T)$ with (arbitrary) periodic points set $\per(X)$, for every pair of open sets $U,V\subset X$, such that $\partial U\setminus  \per(X)\subset V$,  there is an open set $U'$ with
$U \subset U'\subset  U\cup V$, such that $\partial U'$ is the union of a set of zero (discrete) orbit capacity
and a subset of $\per(X)$. This statement should be compared with Theorem \ref{thm:SFBP}. Our proof is not a routine generalization. 
Using the same method, it is very probable one  could prove an analogous statement to the theorem by Lindenstrauss, imposing no condition on periodic orbits.

Burguet \cite[Proposition 2.1]{burguet2019symbolic} proved that a $C^2$-smooth flow without fixed points on a compact smooth manifold (without boundary), satisfying that for any $\tau>0$ the number of periodic orbits of period less than $\tau$ is finite, has the small flow boundary property. A key tool in Burguet's proof is what he calls the \emph{$n$-transverse property} ($\thickapprox$ the \emph{$n$-general property} defined later in this section). This property is established through successive approximation. A crucial point for the approximation scheme to work is the so-called \emph{$C^1$-stability of transversality}, that is if $M_1$ and $M_2$ are compact transverse smooth manifolds then any compact smooth manifolds $\tilde{M}_1$ and $\tilde{M}_2$ which are sufficiently small $C^1$-perturbations of $M_1$ and $M_2$ are transverse (\cite[Corollary A.3.17]{katok1997introduction}. There are two difficulties in generalizing Burguet’s method to topological flows. Firstly transversality is defined in the context of smooth manifolds. Secondly, even if we consider a topological flow on a compact manifold we are faced with the fact that transversality is not \emph{$C^0$-stable}.

Let us give a very rough overview of our proof. We are given $x\in \In S$, where $S$ is a closed cross-section. We may find $U,V\subset S$ with $x\in\In U=U$, $\In V=V$ and $\Pa U\subset V$. Our goal is to find a perturbation of $U$ in the form of $U'\subset S$ with $U\subset U'\subset U\cup V$ and $\Pa U'\subset V$ such that $\cocap(\Pa U')=0$. Actually we establish the stronger property that for every $y\in S$, 
\begin{equation}
\sharp\{t\geq0: \Phi_t(y)\in \Pa U' \}\le d
\end{equation}
(recall that $\dim(X) = d + 1$).
Consider a small neighborhood of $x\in \Pa U'$ and consider the return profile of this neighborhood: i.e., how elements in this neighborhood return to $\Pa U'$.
This can be represented by certain intersections of certain images of this neighborhood. This enables one do use the dimension of this intersection as a proxy to number of returns. Indeed if the dimension drops below zero, the intersection is empty and no further return is possible. The dimension drop mechanism is the key to the proof and this is captured by the concept of  \textit{general position} explained in the next subsection.  

\begin{rem}
    From the proof of Theorem \ref{main thm} (see the beginning of the proof of Lemma \ref{lem:n-general2}), it follows that one may replace the assumption of the countability of periodic orbits by the assumption that the periodic orbits intersect any cross-section in a zero dimension set, in the statement of the theorem. 

\end{rem}
\subsection{General position}

By Lemma \ref{lem:dim_cross-section}, a global closed cross-section in the $(d+1)$-dimensional space $X$ has dimension $d$. We will therefore use the following definition: 

\begin{df}\label{df:genpos}
	A collection $\mathcal{A}$ of subsets in a $d$-dimensional space is said to be in \textbf{general position} if for every finite  $\mathcal{B}\subset \mathcal{A}$, one has that $$\dim(\cap_{C\in \mathcal{B}}C)\le \max\{-1, d-|\mathcal{B}|\}$$
\end{df}
The definition of general position is due to John Kulesza \cite{kulesza1995zero}. To acquire a better understanding, we quote the following sentences from Lindenstrauss \cite{L95}:

\medskip
{\it 
The motivation for this definition is that given a collection of $(d-1)$-dimensional subsets of an $d$-dimensional space then generically any two will have intersection with dimension less than $d-2$, etc.}

\begin{rem}\label{rem:empty_int}
For the purposes of the proof the most important consequence of Definition \ref{df:genpos} is that every finite sub-collection $\mathcal{B}\subset \mathcal{A}$ with $d+1$ elements has \textit{empty intersection}.
\end{rem}

In what follows in this section, we fix two cross-sections $S=S_i$ and $S'=S_i'$ in the complete families $\mathcal{S}$ and $\mathcal{S}'$ for some $1\le i\le N$.  
Recall the definition of $T_{C^n}^{-k}$ in Equation \eqref{eq:T_C_n} and the definition of $I_{C^n}$
in Definition \ref{def:I_C_n} in Subsection \ref{sec:return_names}.
\begin{df}\label{def:n-general}
	Let  $n\ge 0$.
	A set $V\subset S$ is called \textbf{$C^n$-general} for $C^n\in \CC(n)$ if $\left(T_{C^n}^{-k}(\Pa V)\right)_{k\in I_{C^n}}$ is in general position.
It is called \textbf{n-general} if it is $C^n$-general for every $C^n\in \CC(n)$. Moreover, we say that it is \textbf{$C^n$-general} (or \textbf{n-general}) \textbf{on a set $U$} if $\left(T_{C^n}^{-k}(\Pa W) \cap U\right)_{k\in I_{C^n}}$ is in general position (or respectively for all $C^n\in \CC(n)$).
\end{df}
The following remark is easy to establish: 
\begin{rem}\label{rem:gen_post_inherited}
 If a set $V\subset S$ is $n$-general then it is $m$-general for all $0\leq m\leq n$.
\end{rem}

\begin{rem}\label{lem:0-general}
    Note that if $W\subset S$ with $\dim(\Pa W)\leq d-1$, then for $C^0=(C_0)\in \CC(0)$, $I_{C^0}=\{0\}$ and consequently $\dim(\Pa W)\leq  d-1$ implies that 
    $W$ is $0$-general. 
\end{rem}


\subsection{Key Lemma}
The reader may want to recall the notation introduced in Subsection \ref{subsec:notation_3}. In particular,  $X$ is a finite-dimensional space, $\Phi$ is a topological flow on $X$ without fixed points, having a countable number of periodic orbits and $S$ is a member of a  complete family $\mathcal{S}$ of cross-sections. In this subsection, we prove the key technical lemma of the paper:
\begin{lem}\label{lem:n-general2}
		Let $U$ be a subset in $S$ with $\In U=U$ and $n\ge 0$. Then for every subset $V$ in $S$ with $\In V=V$ and $\Pa U\subset V$,  there is a subset $U'\subset S$ with $\In U'=U'$ such that
	\begin{itemize}
		\item [(1)] $U\subset U'\subset U\cup V$;
		\item [(2)] $\Pa U'\subset V$;
		\item [(3)] $U'$ is $n$-general;
		
\item [(4)] $\Pa U'\cap P=\emptyset$.
	\end{itemize} 
\end{lem}

\begin{proof}[Proof of Lemma \ref{lem:n-general2}]

 If $\Pa V=\emptyset$, then we set $U'=U\cup V$. By Lemma \ref{lem:In U=U} and the fact that $\Pa U\subset \In (V) \subset \In (U\cup V)$, we see that $\Pa U'=\emptyset$. Thus $U'$ is $n$-general for all $n\ge 0$. 

Let us assume $\Pa V\not=\emptyset$. We will prove the lemma by induction on $n$. 
Note that by assumption the number of periodic orbits is countable. As a periodic orbit intersects a cross-section a finite number of times, we conclude $P\cap S$ is countable. By Theorem \ref{thm:R3}, the $F_\sigma$ set $P\cap S$ is at most zero-dimensional.
By Lemma \ref{lem:dim_cross-section}, $\dim(S)\leq d$. By Theorem \ref{thm:R2}, one may find a  $F_\sigma$ subset $E$ of $S$ with $\dim(E)\leq 0$  such that $\dim(S\setminus E)\leq d-1$. By Corollary \ref{cor:sum_theorem_for_F_sigma} $\dim(E\cup (P\cap S))\leq 0$. We now treat the base case $n=0$. By Theorem \ref{thm:R1} we may pick for every $x\in \Pa U$ an open neighborhood $U_x\subset V$ such that  $\Pa U_x\cap (P\cup E)=\emptyset$. Let $U_{x_1}, U_{x_2},\ldots, U_{x_m}$ be a cover of $\Pa U$. Define $U'=U\cup \bigcup_{i=1}^m U_{x_i}$. Clearly $\Pa U'\subset \bigcup_{i=1}^m \Pa U_{x_i}$  and therefore Condition (4) holds. In addition it is easy to see that  Conditions (1) and (2) hold. As $\Pa U'\subset S\setminus E$, $\dim(\Pa U')\leq d-1$.   By Remark \ref{lem:0-general}, Condition (3) holds.
 
By induction we suppose that we have constructed $A_0\subset S$ with $\In A_0=A_0$ which satisfies the following properties:
\begin{itemize}
	\item [(A1)] $U\subset A_0\subset U\cup V$.
	\item [(A2)] $\Pa A_0\subset V$.
	\item [(A3)] $A_0$ is $(n-1)$-general\footnote{Recall Remark \ref{rem:gen_post_inherited}.}.
	\item [(A4)] $\CC(n)=\mathcal{F}_0\sqcup \mathcal{F}_1$ such that $\mathcal{F}_1\not=\emptyset$ and $A_0$ is $C^n$-general for all $C^n\in \mathcal{F}_0$.

	\item [(A5)] $\Pa A_0\cap P=\emptyset$.
\end{itemize}
Now fix $\mathcal{D}\in \mathcal{F}_1$. We will construct $U'$ such that
\begin{itemize}
	\item [(B1)] $\In U'=U'$.
	\item [(B2)] $U\subset U'\subset U\cup V$ and $\Pa U'\subset V$.
	\item [(B3)] $U'$ is $(n-1)$-general.
	\item [(B4)] $U'$ is $C^n$-general for all $C^n\in \mathcal{F}_0\sqcup \{ \mathcal{D}\}$.
	\item [(B5)] $\Pa U'\cap P=\emptyset$.
\end{itemize} 
Assuming that such $U'$ has been constructed, we may complete the proof by inducting on $\mathcal{F}_1$, finally obtaining $\mathcal{F}_0=\CC(n)$. Thus it remains to construct $U'$ that satisfies the conditions $(B1)-(B5)$ which will be achieved in the sequel. Indeed set $U'=A_m$ of Lemma \ref{lem:E1-E6v2}.
\end{proof}

\begin{df}
Denote by $B(x, r)$ with $x\in S'$ and $r>0$  the intersection of $ S'$ and the open ball centered at $x$ with radius $r$. Also, denote by $B'$ the \textbf{doubling ball} of $B$, i.e.\ if $B=B(x, r)$ then $B'=B(x, 2r)$.
\end{df}

\begin{lem}\label{lem:D}
Let $A_0$ and $V$ be as in Lemma \ref{lem:n-general2}. There exist open sets $\{B_i\}_{i=1}^{m}$  satisfying the following conditions:
\begin{itemize}
\item [(D1)] for any $C\in \CC(n)$ and for any distinct $j,\ell\in I_{C}$, one has that $T_{C^n}^{-j} (B_i') \cap T_{C^n}^{-\ell}(B_i')=\emptyset$;
	\item [(D2)] $\{B_i\}_{i=1}^{m}$ is an open cover of $\Pa A_0$;
	\item [(D3)] $\bigcup_{i=1}^m B_i\subset V$.
	
\end{itemize} 
\end{lem}

\begin{proof}
Let $x\in \Pa A_0$. We will now construct a ball around $x$, $B_x=B(x,r)$ for some $r>0$, such that $B_x\subset V$ and for any $C\in \CC(n)$ and for any distinct $j,\ell\in I_{C}$, one has that $T_{C^n}^{-j} (B_x') \cap T_{C^n}^{-\ell}(B_x')=\emptyset$.
Since $\mathcal{C}(n)$ is a finite collection, it is sufficient to consider a fixed $C^n=(C_0, C_1, \dots, C_{n-1}, C_n)\in \mathcal{C}(n)$ and fixed $(j,\ell)\in I_{C^n}\times I_{C^n}$ with $j>\ell$ (as one may take the intersection of the open balls associated to each $C^n$ and each $j>\ell$). Assume for a contradiction that for every $p\in \mathbb{N}$  we may find an open neighborhood $W_p$ in $S'$ of $x$ with $\diam(W_p)<\frac{1}{p}$ and $w_p\in X$ such that  $w_p\in T_{C^n}^{-j}W_p\cap T_{C^n}^{-\ell}W_p$. Let $z_p:=T_{C^n}^{j}w_p,\, y_p:=T_{C^n}^{\ell}w_p\in W_p$. This implies we have found sequences $(y_p)_{p\in \N}$ and $(z_p)_{p\in \N}$ in $X$ converging to $x$ such that $T_{D^{j-\ell}}^{j-\ell}y_p=z_p$, where $D^{j-\ell}=(C_{\ell}, \dots, C_{j})\in \mathcal{C}(j-\ell)$. Thus by Proposition \ref{prop:C_ij} there exist $t_p\leq 2\eta (j-\ell)$ such that $\Phi_{t_p}y_p=z_p$.  W.l.o.g.\ $t_p\rightarrow t$, for $k\to \infty$ for some $t\leq 2\eta (j-\ell)$. We conclude that $\Phi_{t}x=x$ which is a contradiction to the fact that $x\notin P$ as (A5) $\Pa A_0 \cap P=\emptyset$. Let now $\{B_i\}_{i=1}^{m}$ be a finite subcover of $\Pa A_0$ extracted from $\{B_x\}$. Clearly properties (D1)-(D3) hold.
\end{proof}

\begin{lem}\label{lem:E1-E6v2}
Let $(B_k)$ and $(B_k')$ be as in Lemma \ref{lem:D}. There exists sets $\{A_k\}_{k=1}^{m}$ such that for $k\geq 1$:
\begin{itemize}
	\item [(E1)] $A_{k-1}\subset A_k$ and $A_k\setminus A_{k-1}\subset B_k'$;
	\item  [(E2)] $\In A_k=A_k$;
	\item [(E3)] $\Pa A_k \subset  V$;
	\item [(E4)]  $\Pa A_k\cap P=\emptyset$;
	\item [(E5)] $\Pa A_k\subset \bigcup_{i=1}^m B_i$;
	\item [(E6)] $A_k$ is
	$(n-1)$-general and $C^n$-general for all $C^n\in \mathcal{F}_0$;
	\item [(E7)] $A_k$ is
	$\mathcal{D}$-general on $\cup_{1\le j\le k} B_j$.
	
\end{itemize}
\end{lem}

\begin{proof}
We will construct sets $\{A_k\}_{k=1}^{m}$ by induction on $k$. As $A_0$ is used as the base case we notice from properties (A1)-(A5) that $A_0$ satisfies (E2)-(E6) and (E7) (in an empty fashion). 

Suppose that  $(A_i)_{i=0}^{k-1}$ have been already defined, satisfying the listed proprieties. Define:
$$
\AA_{C}=\{T_{C}^{-j}(\Pa A_{k-1}): j\in I_{C}\},
$$
for every $C\in \CC(n-1)\cup \mathcal{F}_0\cup \{\mathcal{D}\}$. By Theorem \ref{thm:R2}, for every $\AA\subset \AA_{C}$ with $\bigcap \AA \neq \emptyset$, there exists an $F_\sigma$ subset $E_\AA^C \subset Z_{C}$  of dimension zero such that 
\begin{equation}\label{eq:dim_drop}
\dim(\bigcap \AA \setminus E_\AA^C)=\dim(\bigcap \AA )-1.
\end{equation}
Define
$$
E=\bigcup_{C\in \CC(n-1)\cup \mathcal{F}_0\cup \{\mathcal{D}\}, \AA\subset \AA_{C}, j\in I_{C}} T_{C}^j E_\AA^C,
$$ Clearly, $E$ is a zero-dimensional set in $S'$ by Corollary \ref{cor:sum_theorem_for_F_sigma} and Lemma \ref{lem:F_sigma_invariance}(2). Note that by the inductive assumption $\Pa A_{k-1}\subset \bigcup_{i=1}^m B_i$.  Thus one may find $\epsilon>0$ so that:
$$\overline{\Theta}_{\epsilon}(\overline{B_k}\cap \Pa A_{k-1})\subset \ \bigcup_{i=1}^m B_i,$$
where recall from Subsection \ref{subsec:notation} that the \textit{closed $\epsilon$-tube} around a closed set $Q\subset S'$ is the set $\overline{\Theta}_{\epsilon}(Q):=\{y\in S'|\, \dist(y,Q)\leq \epsilon\}$. 
Clearly $\overline{\Theta}_{\epsilon}(\overline{B_k})\subset B_k'$ for $\epsilon>0$ small enough.
Thus $\overline{\Theta}_{\epsilon}(\overline{B_k}\cap \Pa A_{k-1})\subset B_k'\cap \bigcup_{i=1}^m B_i$.
Note $P\cap S$ is a $F_\sigma$ set with $\dim(P\cap S)\leq 0$.
 By Corollary \ref{cor:sum_theorem_for_F_sigma},  $\dim(E\cup P\cap S)\leq 0$.
By Proposition \ref{prop:avoiding zero_dim}, there exists a set $W$ with $\In W=W$ such that
\begin{itemize}
	\item [(F1)]  $\overline{W}\subset B_k'\cap \bigcup_{i=1}^m B_i$;
	\item [(F2)]  $\overline{\Theta}_{\epsilon}(\overline{B_k}\cap \Pa A_{k-1}) \subset W$;
	\item [(F3)] $\Pa W \cap (E\cup P)=\emptyset$.
	
\end{itemize}
Let $A_k=A_{k-1}\cup W$. We will show that $A_k$ is the set we look for.  By (F1), Conditions (E1) and (E5) hold. Since $\In W=W$, (E2) follows by Lemma \ref{lem:In U=U}. Note: 
\begin{equation}\label{eq:A_k}
\Pa A_k\subset \Pa A_{k-1} \cup \Pa W \subset \Pa A_{k-1} \cup \overline{W}\subset V.
\end{equation}
This gives (E3). By (F3) and (E4) for $k-1$ (E4) holds for $k$. Using \eqref{eq:A_k} and (F3), we have (E7). It remains to check $(E6)$ and $(E7)$. We start by proving $(E6)$. Fix $C\in \CC(n-1)\cup \mathcal{F}_0$. Assume for a contradiction that 
$A_k$ is not $C$-general. We can thus find  $I\subset I_C$ with
\begin{equation*}
\dim( \cap_{j\in I} T_C^{-j} \Pa A_k )>\max \{-1, d-\sharp I\}.
\end{equation*}
By Equation \eqref{eq:A_k} it follows that 
\begin{equation*}
\begin{split}
\cap_{j\in I} T_C^{-j} \Pa A_k 
&\subset \cap_{j\in I} \left( T_C^{-j} \Pa A_{k-1} \cup   T_C^{-j} \Pa W  \right)\\
&\subset \cup_{\alpha\in \{ 0,1\}^{I}}\left( \cap_{j\in I} K_j^{\alpha_j} \right),
\end{split}
\end{equation*}
where $K_j^0=T_C^{-j} \Pa A_{k-1}$ and $K_j^1=T_C^{-j} \Pa W$ for $j\in I$. By $(D1)$ and the fact that $W\subset B_k'$, we see  that $ \cap_{j\in I} K_j^{\alpha_j}=\emptyset$ if $\alpha_j=1$ for at least two $j\in I$. By the induction hypothesis, we have that $\dim(\cap_{j\in I} K_j^0 )\le \max\{-1, d-\sharp I\}$. By Lemmas  \ref{lem:clo_int_open_is_F_sigma} and \ref{lem:F_sigma_invariance} the sets $K_j^{\alpha_j}$ are $F_\sigma$. As the intersection of finitely many $F_\sigma$ sets is an $F_\sigma$ set, we may apply the sum theorem for $F_\sigma$ sets (Corollary \ref{cor:sum_theorem_for_F_sigma}). It follows that there is $\ell\in I$ such that
\begin{equation}\label{eq:dim 1}
\dim(\big (\cap_{j\in I\setminus\{\ell\}} K_j^0 \big )\cap K_\ell^1 )> \max\{-1, d-\sharp I\}.
\end{equation}
Let $\AA=\{K_j^0\}_{j\in I\setminus\{\ell\}}$. Assume there is $x\in K_\ell^1 \cap E_\AA^C$, then $T_C^\ell x\in \Pa W\cap T_C^\ell E_\AA^C$ which is a contradiction to $(F3)$. 
We conclude
$$
\big (\cap_{j\in I\setminus\{\ell\}} K_j^0 \big ) \cap K_\ell^1 \subset \big (\cap_{j\in I\setminus\{\ell\}} K_j^0 \big ) \setminus E_\AA^C.
$$
By the inductive hypothesis (E6) holds for $k-1$ and may be applied on $\cap_{j\in I\setminus\{\ell\}} K_j^0$. Thus, using Equation \eqref{eq:dim_drop}:
\begin{equation*}
\begin{split}
\dim(\big (\cap_{j\in I\setminus\{\ell\}} K_j^0\big ) \cap K_\ell^1 )
&\le  \dim(\big (\cap_{j\in I\setminus\{\ell\}} K_j^0\big )\setminus E_\AA^C)\\
&= \dim(\cap_{j\in I\setminus\{\ell\}} K_j^0)-1\\
&\le  \max\{-1, d-\sharp I\}.
\end{split}
\end{equation*}
This is a contradiction to Equation \eqref{eq:dim 1}. We now prove (E7).  We notice from (F2) and the fact that $A_k=A_{k-1}\cup W$ that $\Pa A_k\cap (B_k\cap \Pa A_{k-1})=\emptyset$, equivalently $\Pa A_k\setminus (B_k\cap \Pa A_{k-1})=\Pa A_k$. In addition we conclude from (F1) that 
$\Pa W \setminus B_k'=\emptyset$.
Thus  $\Pa A_k=\Pa A_k\setminus (B_k\cap \Pa A_{k-1})\subset (\Pa A_{k-1} \cup \Pa W)\setminus (B_k\cap \Pa A_{k-1})\subset (\Pa A_{k-1}\setminus  B_k)\cup (\Pa W\cap B_k')$. Finally:
$$\bigcup_{i=1}^k B_i\cap \Pa A_k\subset(B_{k}\cup\bigcup_{i=1}^{k-1} B_i)\cap \big((\Pa A_{k-1}\setminus  B_k)\cup (\Pa W\cap B_k')\big) $$
The (RHS) equals the union of the four following expressions, one of which is empty:
\begin{itemize}
    \item 
    $B_{k}\cap (\Pa A_{k-1}\setminus  B_k)=\emptyset$,
\item 
$\bigcup_{i=1}^{k-1} B_i\cap (\Pa A_{k-1}\setminus  B_k)\subset \bigcup_{i=1}^{k-1} B_i\cap \Pa A_{k-1}$,
\item 
 $B_{k}\cap \Pa W$ 
\item 
$(\bigcup_{i=1}^{k-1}B_i)\cap (\Pa W\cap B_k') $
\end{itemize}
Fix  $I\subset I_\mathcal{D}$. One can write $\bigcap_{j\in I} T_\mathcal{D}^{-j}(\bigcup_{i=1}^k B_i\cap \Pa A_k)$ as the union of $3^{|I|}$ expressions. Notice that as above each one of these expressions is a finite intersection of $F_\sigma$ sets, thus $F_\sigma$ by itself. Therefore by Corollary \ref{cor:sum_theorem_for_F_sigma},  it is enough to show that each of these expressions has dimension less or equal than $d-|I|$. Let us analyze each of these expressions. If the expression $B_{k}\cap \Pa W$  appears twice, then the expression is empty as  $T_\mathcal{D}^{-j_1}(B_{k})\cap T_\mathcal{D}^{-j_2}(B_{k})=\emptyset$ for $j_1\neq j_2$. The same is true for the expression $(\bigcup_{i=1}^{k-1}B_i)\cap (\Pa W\cap B_k')$ as $T_\mathcal{D}^{-j_1}(B_{k}')\cap T_\mathcal{D}^{-j_2}(B_{k}')=\emptyset$\ for $j_1\neq j_2$.
Thus the expressions involving $\Pa W$ appear at most once. If such expression does not appear at all, then we are only left with the expression $\bigcup_{i=1}^{k-1} B_i\cap \Pa A_{k-1}$ and the dimension estimate is handled by the inductive assumption that condition (E7) holds for $k-1$. Finally we treat the case where an expression involving $\Pa W$ appears exactly once. We thus analyze an expression contained in an expression of the form 
$$T_\mathcal{D}^{-\ell}\Pa W\cap \bigcap_{j\in {I\setminus \{\ell\}}} T_\mathcal{D}^{-j}(\bigcup_{i=1}^{k-1} B_i\cap \Pa A_{k-1})$$
Let $\AA=\{T_\mathcal{D}^{-j}(\Pa A_{k-1}): j\in {I\setminus \{\ell\}}\}$, Assume there is $x\in T_\mathcal{D}^{-\ell}\Pa W \cap E_{\AA}^\mathcal{D}$, then $T_\mathcal{D}^\ell x\in \Pa W\cap T_\mathcal{D}^\ell E_{\AA}^\mathcal{D}\subset \Pa W\cap E$ which is a contradiction to $(F3)$. Thus, using condition (E7) for $k-1$ and Equation \eqref{eq:dim_drop},
\begin{equation*}
\begin{split}
&\dim(T_\mathcal{D}^{-\ell}\Pa W\cap \bigcap_{j\in {I\setminus \{\ell\}}} T_\mathcal{D}^{-j}(\bigcup_{i=1}^{k-1} B_i\cap \Pa A_{k-1}))\\
&\le  \dim(\big(\bigcap_{j\in {I\setminus \{\ell\}}} T_\mathcal{D}^{-j}(\bigcup_{i=1}^{k-1} B_i\cap \Pa A_{k-1})\big )\setminus E_{\AA}^\mathcal{D})\\
&= \dim(\bigcap_{j\in {I\setminus \{\ell\}}} T_\mathcal{D}^{-j}(\bigcup_{i=1}^{k-1} B_i\cap \Pa A_{k-1}))-1\\
&\le  \max\{-1, d-|I|\},
\end{split}
\end{equation*}
as desired. 

\end{proof}

\subsection{Proof of the Main Theorem}

The proof of the main theorem necessitates an approximation lemma. Recall the number $\delta_0$ was defined before Definition \ref{def:delta_0}.
 \begin{lem}\label{lem:open neigh intersection empty}
	Let $A$ and $V$ be subsets such that $A=\overline{A}\subset V=\In V\subset\In S$. Let $n,d\in \N$. Suppose that for all $C^n=(C_0,C_1,\ldots, C_n)\in \CC(n)$ and $I\subset I_C$ with $0\in I$ and $\sharp I=d+1$, it holds, 
	\begin{equation}\label{eq:open neigh intersection empty 0}
\bigcap_{k\in I} T_{C^n}^{-k}A=\emptyset,
	\end{equation}
	Then there exists $U\subset \overline{U}\subset V\subset S_i$ with $\In U=U$ and $A\subset U$ such that 
	\begin{equation}\label{eq:open neigh intersection empty 1}
 Z_{C^n}^{\delta_0}\cap \bigcap_{k\in I} T_{C^n}^{-k}U=\emptyset,
	\end{equation}
	for all $C^n=(C_0,C_1,\ldots, C_n)\in \CC(n)$ and $I\subset I_C$ with $0\in I$ and $\sharp I=d+1$.
\end{lem}
\begin{proof}
	We claim there is $m_0\in \N$, so that for all $m\geq m_0$, $U=U_m:=\Theta_{1/m}(A)$ fulfills the sought after requirements. Assume for a contradiction that this is not the case. We may thus find $ x_m \in Z_{C_m^n}^{\delta_0}$, $C_m^n=(C_{0,m},C_{1,m},\ldots, C_{n,m})\in \CC(n)$, $0=j_0^m<j_1^m<\cdots <j_d^m$, $\{j_i^m\}_{i=0}^d\subset I_{C_m^n}$ so that $T_{C_m^n}^{j_i^m} x_m \in \Theta_{1/m}(A)$ for $i=0,\ldots, d$. As $\CC(n)$ is finite, by passing to a subsequence we may assume w.l.o.g.\ $C_m^n=C^n$ for some fixed $C^n=(C_0,C_1,\ldots, C_n)\in \CC(n)$, and $j_i^m=j_i$ for $i=0,\ldots, d$, for some fixed $I:=\{j_i\}_{i=0}^d\subset I_{C^n}$. By passing to a subsequence we may assume w.l.o.g.\ that $x_m\rightarrow x\in Z_{C^n}^{\delta_0}$. In particular $T_{C^n}^{j_i} x$ is well defined for $i=0,\ldots, d$. Conclude $T_{C^n}^{j_i} x \in A$ for $i=0,\ldots, d$. Thus $x\in \bigcap_{k\in I} T_{C^n}^{-k}A$ which is a contradiction. Q.E.D.
\end{proof}

\begin{thm}\label{thm:SFBP}
	Let $X$ be a compact finite-dimensional space.
    Let $\Phi$ be a topological flow on $X$ without fixed points, having a countable number of periodic orbits. Let $\SS=\{S_i \}_{i=1}^{N}$ be a complete families of cross-sections of $(X,\Phi)$.
    Then for any $U, V\subset S$ with $\In U=U$, $\In V=V$ and $\Pa U\subset V$, there exists $U'\subset S$ with $\In U'=U'$ and $U\subset U'\subset U\cup V$ such that $U'$ has a small flow boundary.
\end{thm}
\begin{proof}
	
	By induction, we will construct $\{U_k\}_{k=0}^\infty$ and $\{V_k\}_{k=0}^{\infty} $ satisfying for all $k\ge 0$,
	\begin{itemize}
		\item [(1)] $U_0=U$ and $V_0=V$;
		\item [(2)] $\In U_k =U_k\subset S$ and $\In V_k=V_k\subset S$;
		\item [(3)]$\Pa U_k\subset V_k$;
		\item [(4)]$\overline{V_{k+1}}\subset V_k$ and $U_k\subset U_{k+1}$; 
		\item [(5)]$U_{k+1}\subset U_k \cup V_k$;
		\item [(6)]for every $C^k=(C_0,C_1,\ldots, C_k)\in \CC(k)$, one has that $ Z_{C^n}^{\delta_0}\cap\bigcap_{j\in I} T_{C^k}^{-j}V_k=\emptyset $ for all $I\subset I_{C^k}$ with $0\in I$ and $\sharp I=d+1$.
	\end{itemize}
 Assuming that such  $\{U_k\}_{k=0}^\infty$ and $\{V_k\}_{k=0}^{\infty}$ have been constructed, we set $U'=\cup_{i=0}^{\infty} U_i$. It is clear that $U'$ is relatively open in $S'$. By $(4)$ and $(5)$, we have that
$$
U_{k+\ell}\cup V_k \subset U_{k+\ell-1}\cup V_{k+\ell-1}\cup V_k \subset U_{k+\ell-1}\cup V_k,
$$
for all $k,\ell>0$. It follows that
\begin{equation*}
U_{k+\ell}\cup V_k \subset  U_{k}\cup V_k,
\end{equation*}
and thus as by (4) $U'=\bigcup_{\ell\geq k} U_\ell$ for all $k$, 
\begin{equation*}
U_{k+1} \subset U'  \subset U_{k+1}\cup V_{k+1},
\end{equation*}
for all $k>0$. Then by $(2), (3), (4)$ 
we obtain that
\begin{equation}\label{eq:last thm 1}
\Pa U' \subset \overline{U_{k+1}\cup V_{k+1}}\setminus U_{k+1} \subset \Pa U_{k+1} \cup \overline{V_{k+1}} \subset V_{k},
\end{equation}
for all $k>0$.
In order to establish $\cocap(\Pa U')=0$, it is sufficient to show that any orbit of the flow hits the set $\Pa U'$ at most $d$ times. Let $x\in \Pa U'\subset S$ and assume for a contradiction that there are 
$0<t_1<t_2<\ldots<t_{d}$ such that $\Phi_{t_i}x\in \Pa U'\subset S$, $i=1,\ldots,d$. 

By Lemma \ref{lem:internal_return}, for some $n$, there is $C^n=(C_0, C_1, \dots, C_{n})\in \CC(n)$ with $x\in  Z_{C^n}^{\delta_0}$ and $I=\{j_i\}_{i=0}^d\subset I_{C^n}$ with $j_0=0<j_1,\ldots <j_d$
such that $T_{C^n}^{j_i}x=\Phi_{t_i}x$ 
for $i=0,\ldots, d$. Thus 
$$x\in Z_{C^n}^{\delta_0}\cap \bigcap_{j\in I} T_{C^n}^{-j}\Pa U'\subset Z_{C^n}^{\delta_0}\cap \bigcap_{j\in I} T_{C^n}^{-j}V_n\neq\emptyset,$$
which is a contradiction to $(6)$.

It remains to construct  $\{U_k\}_{k=0}^\infty$ and $\{V_k\}_{k=0}^{\infty} $. By (1), the sets $U_0$ and $V_0$ are well defined for $k=0$. Suppose that  $(U_i)_{i=0}^k$ and $(V_i)_{i=0}^k $ have been  constructed. Applying Lemma \ref{lem:n-general2} to $U_k,V_k$, we have a subset $U_{k+1}\subset S$ with $\In U_{k+1}=U_{k+1}$ such that
\begin{itemize}
	\item [(A1)] $U_k\subset U_{k+1}\subset U_k\cup V_k$;
	\item [(A2)] $\Pa U_{k+1}\subset V_k$;
	\item [(A3)] $U_{k+1}$ is $(k+1)$-general.
\end{itemize} 
Applying Lemma \ref{lem:open neigh intersection empty} and Remark \ref{rem:empty_int} to $A=\Pa U_{k+1}$, we have a set $V_{k+1}\subset S$ with $\In V_{k+1} = V_{k+1}$ such that
\begin{itemize}
	\item [(B1)] $\Pa U_{k+1}\subset V_{k+1}\subset \overline{V_{k+1}}\subset V_k $;
	\item [(B2)] for all $C^{k+1}=(C_0,\ldots,C_{k+1})\in \CC(k+1)$, it holds that $Z_{C^n}^{\delta_0}\cap \bigcap_{j\in I} T_{C^{k+1}}^{-j}V_{k+1}=\emptyset $ for all $I\subset I_{C^{k+1}}$ with $0\in I$ and $\sharp I=d+1$.
\end{itemize}

Note that the set $U_{k+1}$ and $V_{k+1}$ satisfy the conditions $(2)-(6)$. 
\end{proof}

We can now prove the main theorem:
\begin{proof}[Proof of Theorem \ref{main thm}]
	Fix a point $x\in X$ and a closed cross-section $S$ with $x\in \In S$. Using Lemma \ref{lem:complete family appendix}, one may find two complete families of cross-sections $\SS=\{S_i \}_{i=1}^{N}$ and $\SS'=\{S_i' \}_{i=1}^{N}$ obeying the conditions in the beginning of Section \ref{sec:notations}, with $x\in \In (S_1)\subset S_1 \subset S$. We now choose $U$ and $V$ in $S_1$ with $\In U=U$ and $\In V=V$ such that $x\in U \subset S_1$ and $\Pa U \subset V   \subset S_1$.
	Applying Theorem \ref{thm:SFBP} to $U$ and $V$, we have an open set $U'\subset S_1$ with $\In U'=U'$ and $U\subset U'\subset U\cup V$ such that $U'$ has a small flow boundary and $x\in U'$.  We thus conclude that $(X, \Phi)$ has the small flow boundary property.

\end{proof}

\section{Expansive flows have strongly isomorphic symbolic extensions}\label{sec:Expansive flows have strongly isomorphic symbolic extensions}
\subsection{Background}
In this section, we are interested in expansive flows. We first recall the definition of expansive flows and several properties of such flows. Then we build the symbolic extension of an expansive flow following Bowen and Walters. Finally we give a positive answer to an open question by Bowen and Walters, showing that any expansive topological flow has a strongly isomorphic symbolic extension.

The system $(X,T)$ is said to be {\bf expansive} if there exists an $\delta>0$  such that $\dist(T^nx,T^ny)<\delta$ for all $n\in \Z$ implies $x=y$ where $d$ is a compatible metric on $X$. For example, Anosov diffeomorphisms are expansive (see \cite{anosov1967geodesic}). Keynes and Robertson \cite{keybob} proved that every expansive homeomorphism has a symbolic extension. Moreover, when $X$ is $0$-dimensional, it is conjugate to a subshift. 
The expansive systems have been extensively investigated and many interesting properties have been proven. For example, expansiveness is invariant under topological conjugacy; the \textit{periodic growth}\footnote{The periodic growth of $(X,T)$ is defined as $\limsup_{n\rightarrow\infty}(1/n)\log |P_n|$ where $P_n$ is the set of $n$-periodic points.} and the topological entropy of an expansive discrete system is finite \cite{conze1968points}. Ma\~n\'e \cite{mane1979expansive} showed  that an expansive $\mathbb{Z}$-system must be  finite-dimensional. From this it follows that an expansive $\mathbb{Z}$-system without fixed points has the small boundary property (\cite{L95}). By \cite[Proposition C.1]{burguet2019symbolic} an expansive $\mathbb{Z}$-system with the small boundary property  admits an \textit{essential uniform generator} which induces a strongly isomorphic symbolic extension (\cite[Proposition 3.1]{burguet2019symbolic}). 

\begin{df}
	The flow $(X, \Phi)$ is said to be \textbf{expansive} if for any $\epsilon>0$, there exists $\delta>0$ such that if $\dist(\Phi_t(x), \Phi_{s(t)}y)<\delta$ for all $t\in \R$, for a pair of points $x,y\in X$ and a continuous map\footnote{By \cite[Theorem 3]{bowen1972expansive},  
one can require in the definition of expansiveness  for the function $s$ to be non-decreasing.} $s:\R \to \R$ with $s(0)=0$, then there exists $t\in \R$ with  $|t|< \epsilon$ such that $y=\Phi_t(x)$.
\end{df}
The notion of expansiveness of topological flows was introduced by Bowen and Walters \cite{bowen1972expansive}. 
We remark that the definition of expansiveness is clearly independent of the metric. By \cite[Theorem 5]{bowen1972expansive}, the entropy of an expansive flow is finite and for every $t\geq 0$ the number of periodic orbits of period in $[0,t]$ is finite. By \cite[Lemma 1]{bowen1972expansive}, an expansive flow has only a finite number of fixed points and each of them is an isolated point\footnote{Note that an expansive $\mathbb{Z}$-system has also only a finite number of fixed points. However, these points need not be isolated as witnessed by the (expansive) full shift on two letters.
}. This reduces the study of expansive flows to those without fixed points. Another important property of expansiveness which Bowen and Walters proved is that expansiveness is a conjugacy invariant for flows \cite[Corollary 4]{bowen1972expansive}. One may therefore argue that the definition of expansiveness by Bowen and Walters is the correct one (see also the discussion of other attempts in \cite[Page 180-181]{bowen1972expansive}).

 Keynes and Sears \cite{keynes1981real} extended the above mentioned result of Ma\~n\'e, showing that expansive flows must be finite-dimensional. It can be shown that Anosov flows are expansive (\cite{anosov1967geodesic}). The geodesic flow on a compact smooth manifold of negative curvature is an Anosov flow and thus expansive; also, Axiom A flows are expansive on their nonwandering sets \cite{bowen1972periodic}.

\subsection{Bowen and Walters' construction}\label{sec:Bowen and Walters}
In this section, we recall the construction of Bowen and Walters, i.e.\ the symbolic extension of an expansive flow, following the exposition in \cite[Section 5]{bowen1972expansive}. We start by proving lemmas which do not necessitate assuming expansiveness.  

Let $(X, \Phi)$ be a topological flow without fixed points. By Lemma \ref{lem:complete family appendix}, we may find constants $0<\alpha<\eta$ such that there is a complete family $\mathcal{S}=\{S_i \}_{i=1}^N$  of injectivity time $\eta>0$ with $\diam(S_i)\leq \alpha$ for all $i$ and $\Phi_{[0,\alpha]}\mathcal{G}=\Phi_{[-\alpha, 0]}\mathcal{G}=X$, where $\mathcal{G}=\cup_{1\le i\le N}S_i$. Since $S_i$ are pairwise disjoint and compact, there is $\beta>0$ so that $\Phi_{[0, \beta]}(S_i)$ are pairwise disjoint for all $i$. 
For $x\in \mathcal{G}$, let $r_i(x)$ be the $i$-th return time to $\mathcal{G}$ (see Definition \ref{def:i_th return map}). As $\Phi_{[-\alpha, 0]}\mathcal{G}=X$, every $y\in \mathcal{G}$, has some $x\in \mathcal{G}$ and $0\leq s\leq \alpha$ so that $\Phi_{-s}x=y$, i.e. $\Phi_{s}y=x\in \mathcal{G}$. It follows that $0<\beta\le r_{i+1}(x)-r_i(x)\le \alpha$ for all $x\in \mathcal{G}$. Denote by $\Sigma=\{1,2,\dots, N\}$ and $\sigma$ the (left) shift on $\Sigma^{\mathbb{Z}}$. 
Define
\[
\mathcal{G} \rtimes \Sigma^{\mathbb{Z}}:=\{ (x, \mathbf{q})\in \mathcal{G} \times \Sigma^{\mathbb{Z}}: \exists \mathbf{t}\in \R^{\Z} ~\text{so that}~t_0=0, t_{i+1}-t_i\in [\beta,\alpha], \Phi_{t_i}(x)\in S_{q_i} \}.
\]
The sequence $\mathbf{t}$ appearing in the definition is called a \textbf{return-time sequence} for $(x, \mathbf{q})$; $\mathbf{q}$ is called a \textbf{return-name sequence} for $x$ and $x$ is called a \textbf{realization} of $\mathbf{q}$.
A return-time sequence $\mathbf{t}$ for $x$ is any return-time sequence of the form $(x, \mathbf{q})\in \mathcal{G} \rtimes \Sigma^{\mathbb{Z}}$ for some $\mathbf{q}$.

Note that the sequence $\mathbf{r}=\{r_i(x)\}_{i=-\infty}^{\infty}$ is the \textbf{maximal} return-time sequence for $x$ in the sense that if $\mathbf{t}=\{t_i\}_{i=-\infty}^{\infty}$ is a return-time for $x$, then $\{t_i\}_{i=-\infty}^{\infty}\subset \{r_i(x)\}_{i=-\infty}^{\infty}$. Other return-time sequences than $\mathbf{r}$ for $x$ may exist. This is the case if for example for some $i$, $r_{i+1}-r_{i-1}\leq \alpha$, as one can remove $r_i$ from  $\mathbf{r}$ and still have a return-time sequence for $x$. 
In contrast Lemma \ref{lem:unique_return_time} shows that  $(x, \mathbf{q})$ has a unique return-time sequence. In addition in Lemma \ref{lem:unique return-time}, assuming that  $(X, \Phi)$ is expansive, it is shown that the return-name sequence determines a realization uniquely.

\begin{lem}\label{lem:closed appendix}
Let $(X, \Phi)$ be a topological flow without fixed points. Then $\mathcal{G} \rtimes \Sigma^{\mathbb{Z}}$ is closed in $\mathcal{G} \times \Sigma^{\mathbb{Z}}$.
\end{lem}
\begin{proof}
    Let $(x^n, \mathbf{q}^n)\in \mathcal{G} \rtimes \Sigma^{\mathbb{Z}}$ which converge to $(x, \mathbf{q})$ as $n\to \infty$. Let $\mathbf{t}^n$ be a return-time sequence of $(x^n, \mathbf{q}^n)$. Since $|t_i^n|\le |i|\alpha$ for every $i\in \Z$, we may assume that the limit $\lim_{n\to \infty}t_i^n$ exists, say $t_i$ for each $i\in \Z$. Since $\Sigma$ is finite, we see that for each $i\in \Z$, $q_i^n=q_i$ for large $n>n(i)$. Since $S_{q_i}$ is closed for each $i\in \Z$, we obtain that $\Phi_{t_i}(x)=\lim_{n\to \infty} \Phi_{t_i^n} (x^n)\in S_{q_i}$ for each $i\in \Z$. This implies that $(x,\mathbf{q})\in \mathcal{G} \rtimes \Sigma^{\mathbb{Z}}$ with return-time sequence $\mathbf{t}=(t_i)_{i\in \Z}$.
\end{proof}

\begin{lem}\label{lem:unique_return_time}
Let $(X, \Phi)$ be a topological flow without fixed points. Then for each $(x, \mathbf{q})\in \mathcal{G} \rtimes \Sigma^{\mathbb{Z}}$, the return-time sequence $\mathbf{t}$ is unique.
\end{lem}
\begin{proof}
    Suppose that $\mathbf{t}$ and $\mathbf{t}'$ are two distinct return-time sequences of $(x, \mathbf{q})\in  \mathcal{G} \rtimes \Sigma^{\mathbb{Z}}$. Without loss of generality, we assume that $(t_i)_{i>0}\not= (t_i')_{i>0}$. Let $k=\min\{i>0: t_i\not=t_i' \}$. It follows that $t_{k-1}=t_{k-1}'$ and $\Phi_{t_k}(x), \Phi_{t_k'-t_k}(\Phi_{t_k}(x))\in S_{q_k}$. Since $t_k-t_{k-1}\leq\alpha$ and $t_k'-t_{k-1}'\leq\alpha$, we have that $|t_k'-t_k|\le \alpha<\eta$ and consequently $t_k=t_k'$. This is a contradiction.
\end{proof}
Thus we may define the map:
$$
\mathbf{t}: \mathcal{G} \rtimes \Sigma^{\mathbb{Z}} \to \R^\Z.
$$

\begin{lem}
Let $(X, \Phi)$ be a topological flow without fixed points. Then the map $\mathbf{t}$ is continuous.
\end{lem}
\begin{proof}
    It suffices to show that $t_i: \mathcal{G} \rtimes \Sigma^{\mathbb{Z}} \to [-|i|\alpha,|i|\alpha]\subset \R$ is continuous for each $i\in \Z$. Assume $(x^n, \mathbf{q}^n)\in \mathcal{G} \rtimes \Sigma^{\mathbb{Z}}$ converge to $(x, \mathbf{q})$ as $n\to \infty$. Let $\{t_i(x^{n_j}, \mathbf{q}^{n_j})\}_j$ be an arbitrary converging subsequence. By the argument of Lemma 4.3, $\lim_{j\rightarrow \infty} t_i(x^{n_j}, \mathbf{q}^{n_j})=t_i(x, \mathbf{q})$. This establishes continuity. 
    \end{proof}

From now onward we assume that $(X, \Phi)$ is expansive. We will use the following characterization:

\begin{thm}[\cite{bowen1972expansive}, Theorem 3]\label{Thm:expansive appendix}
    A topological flow without fixed points $(X, \Phi)$ is expansive if and only if for any $\epsilon>0$ there exists $\alpha'=\alpha'(\epsilon)>0$ such that the following holds: if $\mathbf{t}=(t_i)_{i\in \mathbb{Z}}$ and $\mathbf{u}=(u_i)_{i\in \mathbb{Z}}$ with $t_0=u_0=0$, $0<t_{i+1}-t_i\le \alpha'$, $|u_{i+1}-u_i|\le \alpha'$, $t_i\to \infty$ and $t_{-i}\to -\infty$ as $i\to \infty$ and if $x,y\in X$ satisfy $\dist(\Phi_{t_i}(x), \Phi_{u_i}(y))\le \alpha'$ for all $i\in \mathbb{Z}$, then there exists $t\in \R$ with  $|t|\le \epsilon$ such that $y=\Phi_t(x)$.
\end{thm}

W.l.o.g.\ we assume $0<\alpha<\alpha'(\eta)$ as above.
Let $P_2: \mathcal{G} \rtimes \Sigma^{\mathbb{Z}} \to \Sigma^{\mathbb{Z}}$ be the projection on the second coordinate. 

\begin{lem}\label{lem:unique return-time}
Let $(X, \Phi)$ be an expansive topological flow without fixed points. Then the map $P_2$ is injective. 
\end{lem}
\begin{proof}
    Suppose that $(x, {\bf q}), (y, {\bf q}) \in \mathcal{G} \rtimes \Sigma^{\mathbb{Z}}$. Since $(X, \Phi)$ is expansive, by Theorem \ref{Thm:expansive appendix}, we see that $x=\Phi_{t}(y)$ for some $t$ with $|t|\le \eta$. Since $x,y\in S_{q_0}$ and $S_{q_0}$ is a cross-section  of injectivity time $\eta$, we have that $x=y$.
\end{proof}
Let $\mathcal{Z}:=\cup_{1\le i\le N}\Pa (S_i)$ and $\mathcal{W}$ be the set of points whose orbit does not intersect $\mathcal{Z}$. In other words,
\begin{equation}\label{eq:W}
\mathcal{W}=X\setminus \bigcup_{r\in \mathbb{R}} \Phi_{r}(\mathcal{Z}).
\end{equation}
Clearly, the set $\mathcal{W}$ is $\Phi$-invariant, i.e.\ $\Phi_t(\mathcal{W})=\mathcal{W}$ for all $t\in \mathbb{R}$.  
Define 
$$\mathcal{V}=\mathcal{W}\cap \mathcal{G}.$$
 Recall the definition of $r_i(x)$ above. Define $Q: \mathcal{V} \to \Sigma^{\mathbb{Z}}$ by $Q(x)=(Q_i(x))_{i\in \mathbb{Z}}$ where $Q_i(x)=j$ if $\Phi_{r_i(x)}(x)\in S_j$. Note that the map $Q$ sends an element in $\mathcal{V}$ to the return name-sequence associated with its maximal return time-sequence, and  is not necessarily continuous. If $x\in \mathcal{V}$, then also  $\Phi_{r_1(x)}(x)\in \mathcal{V}$, thus $Q(\mathcal{V})$ is $\sigma$-invariant. 
Define $$\Lambda=\overline{Q(\mathcal{V})}.$$
Thus $\Lambda$ is the closure of the set of return name-sequences associated with the maximal return time-sequences of elements in $\mathcal{V}$. Clearly  $(\Lambda, \sigma)$ is a subshift. 
As the map $P_2$ is continuous and injective, its inverse $P_2^{-1}: P_2(\mathcal{G} \rtimes \Sigma^{\mathbb{Z}}) \to \mathcal{G} \rtimes \Sigma^{\mathbb{Z}}$ is continuous.  Clearly, for $x\in \mathcal{V}$, the point $(x, Q(x))\in \mathcal{G} \rtimes \Sigma^{\mathbb{Z}}$ and consequently $Q(x)\in P_2(\mathcal{G} \rtimes \Sigma^{\mathbb{Z}})$. Since $P_2(\mathcal{G} \rtimes \Sigma^{\mathbb{Z}})$ is compact and contains $Q(\mathcal{V})$, we see that $\Lambda$, which is the closure of $Q(\mathcal{V})$, is contained in $P_2(\mathcal{G} \rtimes \Sigma^{\mathbb{Z}})$.
Let $P_1: \mathcal{G} \rtimes \Sigma^{\mathbb{Z}} \to \mathcal{G}$ be the projection on the first coordinate. Define the continuous map
$P: \Lambda \to \mathcal{G}$ by 
$$P({\bf q})=P_1(P_2^{-1}({\bf q}))$$ for ${\bf q} \in \Lambda$. Thus, $P$ associates with a return-name sequence in $\Lambda$ its realization in $\mathcal{G}\subset X$. Clearly, $P\circ Q$ is the identity on $\mathcal{V}$.

We define the continuous function $f: \Lambda \to [\beta, \alpha]$ by 
$$
f({\bf q})=t_1(P_2^{-1}{\bf q}).
$$

\begin{rem}\label{rem:f}
 From the proof of Lemma \ref{lem:unique_return_time} it is clear that $f({\bf q})$ is the first positive time when the realization of $\bf q$  in $\mathcal{G}$ returns to  $S_{q_1}$ but it is possible that it returns to $\mathcal{G}$ earlier. However, if $\bf q\in Q(\mathcal{V})$ then $f({\bf q})$ is the first positive return time to $\mathcal{G}$ of the realization of $\bf q$ (due to the definition above of $Q$ involving maximal return time-sequences).
\end{rem}

\begin{df}
Denote by $(\Lambda_{f}, \Psi)$ the suspension flow over $(\Lambda, \sigma)$. Define $\pi: \Lambda_f \to X$ by
$$
\pi(({\bf q}, t))=\Phi_t(P({\bf q})).
$$
\end{df}
\begin{thm}[Theorem 10, \cite{bowen1972expansive}]\label{thm:symbolic appen 1}
   Let $(X, \Phi)$ be an expansive topological flow without fixed points. Then the flow $(X, \Phi)$ is a factor of $(\Lambda_{f}, \Psi)$ via $\pi$. 
\end{thm}
\begin{proof}
First, we show that $\pi$ is well defined, i.e.\ that it holds $\pi(\q, f(\q))=\pi(\sigma(\q), 0)$. This implies easily that $\pi$ is $\R$-equivariant. By the definition of $\pi$, one has to show $\pi(({\bf q}, f({\bf q})))=\Phi_{f(\q)}(P(\q))=P(\sigma(\q))=\pi((\sigma({\bf q}), 0))$. It is enough to show  $(\Phi_{f(\q)}(P(\q)), \sigma(\q))\in \mathcal{G} \rtimes \Sigma^{\mathbb{Z}}$ as the return-name sequence determines the realization. Indeed, we claim that the associated return-time sequence of $(\Phi_{f(\q)}(P(\q)), \sigma(\q))$ is $\{{t}_{i+1}(P(\q), \q)-t_1(P(\q), \q)\}_{i\in \Z}$, which is confirmed by:
    $$
    \Phi_{{t}_{i+1}(P(\q), \q)-t_1(P(\q), \q)} (\Phi_{f(\q)}(P(\q)))=\Phi_{{t}_{i+1}(P(\q), \q)} (P(\q)) \in S_{q_{i+1}},
    $$
    for $i\in \Z$, where we used that $t_1(P(\q), \q)=f(\q)$ (see Remark \ref{rem:f}).

Finally, we note that the image $\pi(\Lambda_f)$ is a $\Phi$-invariant closed subset of $X$ which contains $\mathcal{V}$ as for $x\in \mathcal{V}$ it holds $\pi((Q(x),0))=P\circ Q(x)=x$. 

From Definition  \ref{def:Flow boundaries and interiors} it is easy to see that the closed set $\Phi_{[-\eta, \eta]}\mathcal{Z}$ has empty interior (see also \cite[Lemma 2.4(1)]{burguet2019symbolic}). By Baire category theorem this implies that  $\mathcal{W}=\bigcup_{r\in \mathbb{R}} \Phi_{r}(\mathcal{V})
$ is dense in $X$. We conclude that $\pi(\Lambda_f)=X$.
\end{proof}

The following lemma is crucial for the proof of Theorem \ref{thm:symbolic appen} in the sequel.

\begin{lem}\label{lem:one-to-one appendix}
Let $(X, \Phi)$ be an expansive topological flow without fixed points. Let $(\q,t)\in \Lambda_f$ for some $0\leq t<f(\q)$. If $\pi((\q,t))=z\in \mathcal{V}$, then $t=0$ and $\q=Q(z)$.
\end{lem}
\begin{proof}
By definition $\pi(\Lambda_{f}\times \{0\})\subset \GG$. Thus  $x:=\Phi_{-t}(z)\in \mathcal{G}\cap \mathcal{W}=\mathcal{V}$.  Moreover by the definition and equivariance of $\pi$, $x=\pi(\q,0)=P(\q)$. We will show $\q=Q(x)$. 
Note that as $z=\Phi_{t}(x),x\in \mathcal{V}$, $t$ is a return-time of $x$ to $\mathcal{G}$. By assumption $0\leq t<f(\q)$ and $f(\q)$ is the first positive return time to $\mathcal{G}$ (see Remark \ref{rem:f}). Thus $\q=Q(x)$ implies that $t=0$ and from this it holds as desired, that $\q=Q(x)=Q(\Phi_{t}(x))=Q(z)$. 
To show $\q=Q(x)$, pick $x_n\in \mathcal{V}$ such that $Q(x_n)\to \q$ as $n\to \infty$.
As $P$ is continuous and $(P\circ Q)_{|\mathcal{V}}=\id$, we see that $x_n=P(Q(x_n))\to P(\q)=x$ as $n\to \infty$. Let ${\bf t}={\bf t}(x, Q(x))$. For $N>0$, let 
    $$
    E_N=\bigcup_{i=-N}^{N-1} [t_i+\beta/2, t_{i+1}-\beta/2]
    $$
    and define the open set
    $$
    O_N=\left( \bigcap_{i=-N}^{N} \Phi_{-(t_i-\beta/2, t_i+\beta/2) }\In(S_{Q_i(x)})\right) \setminus \Phi_{-E_N} \mathcal{G}.
    $$
Note that for $i=-N,\ldots, N$, $$\Phi_{(t_i-\beta/2, t_i+\beta/2) }x\in Q_i(x) \text{ and }x\notin \Phi_{[t_i+\beta/2, t_{i+1}-\beta/2]}\GG.$$   As $x\in\mathcal{V}$, it holds that for all $t\in\R$, $\Phi_t x\notin\mathcal{Z}$. Thus, if $\Phi_t x\in S_i$, then in effect $\Phi_t x\in\In S_i$. We conclude $O_N$ is open and contains $x$.
    
    Note that if $y\in O_N\cap \mathcal{V}$, then $Q_i(y)=Q_i(x)$ for $-N\le i\le N$. It follows that for each $N$, there exists $m(N)$ such that for $n>m(N)$, $Q_i(x_n)=Q_i(x)$ for $-N\le i\le N$. This implies that $Q(x_n)\to Q(x)$ as $n\to \infty$.  Recall that $Q(x_n)\to \q$ as $n\to \infty$. Thus we obtain that $\q=Q(x)$ as desired.
\end{proof}

\subsection{Bowen and Walters' question answered}
Bowen and Walters asked the following question:

\begin{question}(\cite[Problem, p. 192]{bowen1972expansive})
Let $(X,\Phi)$ be an expansive topological flow without fixed points. Can the symbolic extension $\pi$ in Theorem \ref{thm:symbolic appen 1} be made entropy-preserving by choosing carefully the cross-sections $\{S_i\}_{i=1}^{N}$?
\end{question}


 
\begin{thm}\label{thm:symbolic appen}
    Let $(X,\Phi)$ be an expansive flow without fixed points. Then $\pi$ is strongly isomorphic.
\end{thm}
\begin{proof}
   By \cite[Theorem 4.2]{keynes1981real} an expansive flow is finite-dimensional. By \cite[Theorem 5]{bowen1972expansive}, an expansive flow has the property that for any $\tau>0$, the number of periodic orbits of period less than $\tau$ is finite. Since $(X, \Phi)$ has no fixed points, we can now apply Main Theorem \ref{main thm} to conclude $(X, \Phi)$ has the small flow boundary property. Adapting the notation and results of Subsection \ref{sec:Bowen and Walters}, we notice that the cross-sections $S_i$ in Lemma \ref{lem:complete family appendix} may be chosen such that the $S_i$ have small flow boundaries for each $i$. Thus  $\Phi_{[-\eta,\eta]}(\mathcal{Z})$ is a null set and by $\sigma$-additivity of measures, $\mathcal{W}$ defined in Equation \eqref{eq:W} is a full set. It therefore suffices to show that $\pi$ is one-to-one when restricted to $\pi^{-1}(\mathcal{W})$.
    In other words, if $\pi(({\bf q_1}, t_1))=\pi(({\bf q_2}, t_2))=y\in \mathcal{W}$, then $({\bf q_1}, t_1)=({\bf q_2}, t_2)$. Write $y=\Phi_{t}x$ for some $t\in \R$ and $x\in \mathcal{V}=\mathcal{W}\cap \mathcal{G}$. 
    By applying $\Phi_{-t}$, we may assume w.l.o.g.\ $\pi(({\bf q_1}, t_1))=\pi(({\bf q_2}, t_2))\in \mathcal{V}$. It is thus enough to show that $\pi$ is one-to-one when restricted to $\pi^{-1}(\mathcal{V})$. This follows from Lemma \ref{lem:one-to-one appendix}.
\end{proof}

The following theorem gives a strong positive answer to Bowen and Walters' question above.

\begin{thm}\label{thm:thm B} (=Theorem B) Let $(X, \Phi)$ be an expansive flow. Then it has a strongly isomorphic symbolic extension.
\end{thm}
\begin{proof}
If $(X, \Phi)$ has no fixed points then the result follows from Theorem \ref{thm:symbolic appen}. If $(X, \Phi)$ has fixed points then they are isolated (\cite[Lemma 1]{bowen1972expansive}), so the result follows from the previous case.
\end{proof}

\begin{rem}
  It is possible to strengthen the above theorem by achieving a strongly isomorphic symbolic extension $\pi:(\Lambda_{f}, \Psi)\rightarrow (X, \Phi)$ which is at same time uniformly finite-to-one, that is there exists $K>0$ so that for all $x\in X$, $|\pi^{-1}(x)|\leq K$.
The proof will be included in a forthcoming work.  
\end{rem}

\section{Appendix}\label{sec:Appendix}

\subsection{Existence of a complete family}

The following lemma is obtained by a slight modification of the proof of Lemma $7$ of \cite{bowen1972expansive}. See also  \cite[Lemma 2.4]{keynes1981real} for a similar construction. 
\begin{lem}\label{lem:complete family appendix}
Let $(X, \Phi)$ be a topological flow without fixed points. There is an $\eta>0$ so that the following holds. For each $\alpha>0$, $z\in X$ and cross-section $S$ with $z\in \In (S)$, there are two finite families $\mathcal{S}=\{S_i \}_{i=1}^N$ and  $\mathcal{S}'=\{S_i' \}_{i=1}^N$ of pairwise disjoint closed cross-sections  of injectivity time $\eta$ and diameter at most $\alpha$ so that 
	\begin{itemize}
		\item $z\in \In(S_1)\subset S'_1\subset S$,
		\item $\overline{S_i}\subset \In(S_i')$ for all $1\le i\le N$;
		\item $\Phi_{[0,\alpha]}\mathcal{G}=\Phi_{[-\alpha, 0]}\mathcal{G}=X$,
	\end{itemize}
	where $\mathcal{G}=\cup_{1\le i\le N}S_i$. 
\end{lem}
\begin{proof}
	By Theorem \ref{thm:Whitney}, for each $x\in X$ there is a cross-section $S_x$  of injectivity time $2\eta_x>0$ such that $x\in \In S_x$. By compactness of $X$, there are $x_i\in X$ $(1\le i\le n)$ with $x_1=z$ and $S_{x_1}=S$ such that
	$$
	X=\cup_{i=1}^n \Phi_{(-\eta_{x_i}, \eta_{x_i})}\In S_{x_i}.
	$$
	Let $\eta=\min_{1\le i\le n}\{\eta_{x_i} \}$. Then for each $x$ there is an $x_i$ and an $\rho_x\in (-\eta_{x_i}, \eta_{x_i})$  with $x\in \Phi_{\rho_x}\In S_{x_i}$. Let $T_x:=\Phi_{\rho_x}S_{x_i}$ which is a cross-section  of injectivity time at least $\eta_{x_i}\ge 2\eta$ and $x\in \In T_x$.
	
	Let $\alpha>0$ be given. Choose $\epsilon>0$ sufficiently small such that $\epsilon\le \min\{\alpha/4, \eta \}$ and $\diam(\Phi_r (A))\leq\alpha$ whenever $|r|\le \epsilon$ and $A\subset X$ with $\diam(A)<\epsilon$. For each $x\in X$, let $V_x, V_x'\subset \In T_x$ be closed neighborhoods of $x$ in $T_x$ with $V_x\subset \In V_x'$ and  $\diam(V_x')<\epsilon$. 
	Then $x\in \In V_x\subset V_x'$ and $V_x, V_x'$ are cross-sections  of injectivity time $2\eta$ and diameter at most $\alpha$. As $X$ is compact, we can find $y_i\in X, 1\le i\le k$ with $y_1=z$ such that
	$$
	X=\cup_{i=1}^k \Phi_{(-\epsilon, \epsilon)}\In V_{y_i}.
	$$  
	We construct finite pairwise disjoint families $\mathcal{S}_i$ and $\mathcal{S}_i'$ of closed cross-sections recursively. Let $\mathcal{S}_1=\{ V_{y_1}\}$ and $\mathcal{S}_1'=\{V_{y_1}' \}$. Suppose $\mathcal{S}_{i-1}$ and $\mathcal{S}_{i-1}'$ have been defined. For each $y\in V_{y_i}'$, $\Phi_{[-\epsilon, \epsilon]}(y)\cap \cup_{S'\in \mathcal{S}'_{i-1}}S'$ is a finite set of points since $\mathcal{S}'_{i-1}$ consists of cross-sections. As $\Phi$ is continuous and $\cup_{S'\in \mathcal{S}'_{i-1}}S'$ is closed, there exists a non-empty open interval $I_y\subset (-\epsilon, \epsilon)$ and closed neighborhoods $W_y, W_y'$ of $y$ in $V_{y_i}$ and $V_{y_i}'$ respectively with $W_y\subset \In W_y'$ such that
	\begin{equation}\label{eq:empty_int}
	\Phi_{I_y}(W_y')\cap \cup_{S'\in\mathcal{S}'_{i-1}}S'=\emptyset    \end{equation}

Let $y^1, y^2, \dots, y^\ell$ be  points in $V_{y_i}'$ such that $W_{y^1}, \dots, W_{y^\ell}$ cover $V_{y_i}$ and $W_{y^1}', \dots, W_{y^\ell}'$ cover $V_{y_i}'$. Pick distinct $\rho_1\in I_{y^1}, \dots, \rho_\ell\in I_{y^\ell}$ and set
	$$
	\mathcal{S}_i=\mathcal{S}_{i-1} \cup \{\Phi_{\rho_1}(W_{y^1}), \dots, \Phi_{\rho_\ell}(W_{y^\ell})  \}
	$$
	and 
	$$
	\mathcal{S}_i'=\mathcal{S}_{i-1}' \cup \{\Phi_{\rho_1}(W_{y^1}'), \dots, \Phi_{\rho_\ell}(W_{y^\ell}')  \}.
	$$
	By Equation \eqref{eq:empty_int}
and as $\rho_1, \dots, \rho_\ell$ are distinct,  the members of $\mathcal{S}_i$ are pairwise disjoint. The same holds for $\mathcal{S}'_i$. Set $\mathcal{S}=\mathcal{S}_k$ and $\mathcal{S}'=\mathcal{S}_k'$. Clearly $X= \Phi_{[-2\epsilon, 2\epsilon]}\mathcal{G}$. For every $x\in X$, $\Phi_{2\epsilon}x\in \Phi_{[-2\epsilon, 2\epsilon]}\mathcal{G}$. Thus $x\in \Phi_{[-4\epsilon, 0]}\mathcal{G}\subset\Phi_{[-\alpha, 0]}\mathcal{G}$. Thus $X=\Phi_{[-\alpha, 0]}\mathcal{G}$. Similarly $X=\Phi_{[0, \alpha]}\mathcal{G}$.
	 Q.E.D.
\end{proof}



\bibliographystyle{alpha}
\bibliography{universal_bib}

\end{document}